\documentclass[12pt]{article}
\usepackage[english]{babel}
\usepackage{graphicx}
\usepackage{amsmath,amsthm,amssymb}
\usepackage{times}
\usepackage{enumerate}
\usepackage{amsmath}
\usepackage{amsfonts}
\usepackage{pdfsync}
\usepackage{cleveref}
\usepackage{autonum}
\usepackage{enumitem}
\usepackage[usenames, dvipsnames]{color}

\newcommand{\R}{{\mathbb R}}
\newcommand{\N}{{\mathbb N}}
\renewcommand{\L}{\Lambda}

\DeclareMathOperator{\Dom}{Dom}

\DeclareMathOperator{\Lip}{Lip}
\DeclareMathOperator{\dist}{dist}

\def\titlerunning#1{\gdef\titrun{#1}}
\makeatletter
\def\author#1{\gdef\autrun{\def\and{\unskip, }#1}\gdef\@author{#1}}
\def\address#1{{\def\and{\\}\renewcommand{\thefootnote}{}%
\footnote {#1}}%
\markboth{\autrun}{\titrun}}
\makeatother
\def\email#1{e-mail: #1}
\def\subjclass#1{{\renewcommand{\thefootnote}{}%
\footnote{\emph{Mathematics Subject Classification (2010):} #1}}}
\def\keywords#1{\par\medskip
\noindent\textbf{Keywords.} #1}

\theoremstyle{definition}
\newtheorem{theorem}{Theorem}[section]
\newtheorem{corollary}[theorem]{Corollary}
\newtheorem{lemma}[theorem]{Lemma}

\newtheorem{proposition}[theorem]{Proposition}
\newtheorem*{theorem*}{Theorem}
\usepackage{amsthm}

\theoremstyle{definition}
\newtheorem{definition}[theorem]{Definition}
\newtheorem{remark}[theorem]{Remark}
\newtheorem{example}[theorem]{Example}

\newtheorem*{hypothesisH1'}{Hypothesis $\mathbf{(H_2')}$}
\newtheorem*{hypothesisH1''}{Hypothesis ($\mathcal S^{\infty}_{x_*}$)}
\newtheorem*{hypothesisS'}{Hypothesis ($\mathcal S^{\infty}_{x_*}$)}
\newtheorem*{hypothesisH'}{Hypothesis ($\mathcal A^{\infty}_{x_*}$)}

\newtheorem*{problemP'}{Problem \textbf{(P$'$)}}

\newtheorem*{conditionS}{Condition (S)}
\newtheorem*{conditionD}{Condition (D)}

\newtheorem*{basicass}{Basic Assumptions}
\newtheorem*{structass}{Structure Assumptions}

\numberwithin{equation}{section}

\begin{document}
\titlerunning{Avoidance of the gap}
\title{Avoidance of the Lavrentiev gap for one-dimensional non autonomous functionals with state constraints}
\author{
 Carlo Mariconda
}
\maketitle
\date
\address{
Carlo Mariconda (corresponding author), ORCID $0000-0002-8215-9394$),  Universit\`a
degli Studi di Padova,  Dipartimento di Matematica  ``Tullio Levi-Civita'',   Via Trieste 63, 35121 Padova, Italy;  \email{carlo.mariconda@unipd.it}}
\subjclass{49, 49J45, 49N60}
\keywords{Regularity, Lipschitz, minimizing sequence, approximation, radial convexity, control, effective domain, extended valued, gap
}
\begin{abstract}
Let $F(y):=\displaystyle\int_t^TL(s, y(s), y'(s))\,ds$ be a positive functional (the ``energy''), unnecessarily autonomous,   defined on the space of Sobolev functions $W^{1,p}([t,T]; \R^n)$ ($p\ge 1$). We consider the problem of minimizing $F$ among the functions $y$ that possibly satisfy one, or both, end point conditions.
In many applications, where the lack of regularity or convexity or growth conditions does not ensure the existence of a minimizer of $F$, it is important to be able to approximate the value of the infimum of $F$ via a sequence of Lipschitz functions satisfying the given boundary conditions. Sometimes, even with some polynomial, coercive and convex Lagrangians in the velocity variable, thus ensuring the existence of a minimizer in the given Sobolev space, this is not  achievable: this fact is know as the Lavrentiev phenomenon.
The paper deals on the avoidance of the Lavrentiev phenomenon under  the validity  of a further given state constraint of the form $y(s)\in\mathcal S\subset\R^n$ for all $s\in [t,T]$. \\
Given $y\in W^{1,p}([t,T];\R^n)$ with $F(y)<+\infty$ we give a constructive recipe for building a sequence $(y_h)_h$ of Lipschitz reparametrizations of $y$, sharing with $y$ the same boundary condition(s), that converge in energy to $F(y)$. With respect to  previous literature on the subject,  we distinguish the case of (just) one end point condition from that of both,  enlarge the class of Lagrangians that satisfy the sufficient conditions and show that $(y_h)_h$ converge also in $W^{1,p}$ to $y$. Moreover, the results apply also to extended valued Lagrangians  whose effective domain is bounded.
The results gives new clues even when the Lagrangian is autonomous, i.e., of the form $L(s,y,y')=\L(y,y')$. The paper follows two recent papers \cite{MTrans, CM5} of the author on the subject.
\end{abstract}
\section{Introduction}
We consider here a one-dimensional, vectorial functional of the calculus of variations
\[F(y)=\int_t^TL(s, y(s), y'(s))\,ds\]
defined on the space of Sobolev functions $W^{1,p}(I, \R^n)$ on $I:=[t,T]$ with values in $\R^n$, for some $p\ge 1$. In the paper the {\em Lagrangian} $L(s,y,u)$ is Borel and is assumed to have values in $[0, +\infty[\cup\{+\infty\}$. Following the terminology of control theory we will refer to $s$ as to the {\em time} variable,  to $y$ as to the {\em state} variable and to $u$ as to the {\em velocity} variable.
Unless there is a minimizer, it may be desirable to approximate the infimum of the {\em energy} $F$ through  a sequence of values $F(y_h)$ along a sequence $(y_h)_h$ of Lipschitz functions, that possibly share some desired boundary values and constraints. Though the Lipschitz functions are dense in $W^{1,p}(I; \R^n)$, unless some a priori growth assumptions from above are satisfied, this approximation is not always possible; this fact is summarized saying that the Lavrentiev phenomenon occurs.
Of course, the gap does not occur if the minimizers exist and are Lipschitz: we refer the interested reader to   \cite{CVTrans, Clarke1993, DMF,  Cellina, CF1, MTLip, BM2}.
Conditions ensuring the non-occurrence of the phenomenon, other than Lipschitz regularity of the minimizers, were established by Alberti and Serra Cassano in \cite{ASC} in the {\em  autonomous case} (i.e., $L(s,y,u)=L(y,u)$) where they showed that, for the problem with {\em one end point} condition, the phenomenon does not occur under a suitable local boundedness condition. When the Lagrangian is non autonomous, the gap may occur (though being quite rare if $L$ is coercive, see \cite{Zasl}), even with innocent looking Lagrangians, like  in Manià's \cite{Mania} problem
\[\min F(y):=\int_0^1(y^3-s)^2(y')^6\,ds:\, y\in W^{1,1}(I),\, y(0)=0,\, y(1)=1.\]
Several results concerning the {\em non autonomous} case are present in the literature (see \cite{Loewen, TZ, Zas}), but require some additional regularity of the Lagrangian on the state variable, not present in the autonomous case: as Carlson shows in \cite{Carl} many do actually derive from Angell -- Cesari Property ($\mathcal D$) introduced in  \cite{CeAng}.
An extension of \cite[Theorem 2.4]{ASC} for  non autonomous Lagrangians of the form $L(s,y,u)=\L(s,y,u)\Psi(s,y)$, without requiring further regularity on the state or velocity variables of $\L$,  was recently provided by the author in \cite{CM5} where  it is pointed out that the conditions that ensure the non-occurrence of the gap for problems with just {\em one end point} condition may, even in the autonomous case, not be sufficient for those with {\em both} end  point conditions: the difficulty of preserving the  boundary condition was noticed  also in the multidimensional setting (see \cite{BMT, MT2020, MT2020Open}). In this situation \cite[Theorem 3.1]{CM5} introduces a new sufficient condition conjectured by Alberti, covering mostly the case of {\em real valued} Lagrangians or when the {\em effective domain} of $\L$ (i.e., the set where $\L$ is finite) contains an {\em unbounded} rectangle.

Another direction was followed by Cellina and his collaborators  Ferriero, Marchini, Treu, Zagatti in \cite{CTZ, CF1} for {\em non autonomous problems} and,  in \cite{CFM} with Marchini,  for Lagrangians of the form $L(s,y,u)=\L(y,u)\Psi(s,y)$ under an additional {\em convexity hypothesis} on $u\mapsto\L(y,u)$, continuity of $\L, \Psi$. Here, given $y\in W^{1,p}(I;\R^n)$ with $F(y)<+\infty$, the convexity assumption on $\L$ allows to build a sequence of Lipschitz {\em reparametrizations} $(y_h)_h$ of $y$ with the desired end point conditions $y_h(t)=y(t), y_h(T)=y(T)$ and  such that $\limsup_{h} F(y_h)\le F(y)$; in particular since $y_h(I)=y(I)$ for all $h$, the sequence preserves possible state constraints. This  reparametrizations technique was the key tool in \cite{MTrans} to establish 
the non  occurrence of the Lavrentiev phenomenon for the problem with two end point constraints followed in \cite[Corollary 5.7]{MTrans} for non autonomous Lagrangians $L(s,y,u)=\L(s,y,u)$ assuming,  in the real valued case:
\begin{itemize}
\item A {\em local Lipschitz condition} (named (S) by many authors) on $s\mapsto\L(s,y,u)$. Property (S) is known to be a sufficient condition  for the validity of the Du Bois-Reymond equation; we refer to \cite{Cesari} for the smooth case,  to \cite{Clarke1993} for the nonsmooth convex case under weak growth assumptions,  to \cite{BM1, BM2} by Bettiol and the author in the general case.
\item {\em Radial convexity} on $\L$ in the velocity variable, i.e., convexity of $0<r\mapsto \L(s,y, ru)$. The role of radial convexity in Lipschitz regularity was shown by the author jointly with  Treu in \cite{MTLip} for the autonomous case, with Bettiol in \cite{BM1, BM2, BM3} in the general case and for some optimal control problems.
\item A  suitable local boundedness conditions.
\item A {\em linear growth condition} from below of the form $\L(s,y,u)\ge \alpha|u|-d$, $\alpha>0$.
\end{itemize}
In the extended valued case, the same conclusion was obtained under the additional requirement that {\em the limit of $\L$ at the boundary of the domain is $+\infty$},  giving new light in the case where the effective domain of $\L$ is bounded.
The above Condition (S)  and the local boundedness condition are essential to establish the non-occurrence of the gap.
The emphasis in \cite{MTrans} was given to establish not only the non-occurrence of the gap, but even the existence of  {\em equi}-Lipschitz  minimizing sequences, with some {\em uniformity} in the initial time and datum,  assuming a very mild growth condition from below introduced by Clarke in \cite{Clarke1993}. In view of the results of \cite{CM5} some questions concerning the results of \cite{MTrans} arise:
\begin{enumerate}
\item May one extend the results of \cite{MTrans} to the class of Lagrangians of the form $\L(s,y,u)\Psi(s,y)$ considered in \cite{CM5}, with $\L$ alone satisfying (S)? Notice that this class il strictly larger than the one of functions $\L(s,y,u)$ satisfying (S). For instance $L(s,y,u):=\L(s,y,u)\Psi(s,y)$ with
\[\forall s\in [0,1], \forall y,u\in\R\qquad \L(s,y,u):=u^2,\, \Psi(s,y)=\sqrt s\]
does not satisfy (S), but $\L$ does.
\item Which of the conditions formulated in \cite[Corollary 5.7]{MTrans} are really needed for the {\em one end point} condition problem?
\item Can the reparametrization argument work used in the proof of \cite[Theorem 5.1]{MTrans} be adapted for the more general class of Lagrangians considered in \cite{CM5}? Is radial convexity w.r.t. the velocity variable sufficient for the purpose?
\item Can one get rid of the growth assumptions from below needed in \cite{MTrans}?
\end{enumerate}
All of the above questions have an answer in Theorem~\ref{thm:Lav2}  and Corollary~\ref{coro:Lav3}; in particular the answer to Questions 1, 2 and 3 are positive. Moreover, with respect to \cite{MTrans},  we weaken a local boundedness condition in the spirit of \cite[Proposition 3.15]{BM4}, and show that a linear growth from below (see \S~\ref{sect:linear}) is just a desirable, though unnecessary option.
Moreover, it turns out that the limit condition on $\L$ at the boundary of the domain is not needed for the one end point condition problem.
A new fact with respect to the previous literature \cite{CFM, CF1, MTrans} based on the reparametrization method is, given $y\in W^{1,p}(I;\R^n)$, the convergence of the built Lipschitz approximating sequences $(y_h)_h$ not only in energy, but also in $W^{1,p}$ norm, to $y$. The method  is   constructive: in Example~\ref{ex:alberti}  we show how to build an explicit suitable Lipschitz approximating sequence following the recipe of the proof of Theorem~\ref{thm:Lav2}.
A discussion on the choice of alternative distance-like functions, other than the Euclidean one,  is outlined in \S~\ref{sect:Distance} in order to include a more ample class of  extended valued Lagrangians.

\section{Notation and Basic assumptions}
\subsection{Basic Assumptions}
Let $p\ge 1$. The functional $F$ (sometimes referred as to the ``energy'')  is defined by
\[\forall y\in W^{1,p}(I,\R^n)\qquad F(y):=\int_IL(s, y(s), y'(s))\, ds,\]
where $L(s,y,v)$ is of the form $L(s,y,v)=\L(s, y, v)\Psi(s,y)$.
\begin{basicass}
We assume  the following conditions.
\begin{itemize}
\item $I=[t, T]$ is a closed, bounded  interval of $\R$;
\item $\L:I\times \R^n\times\R^n\to [0, +\infty[\cup\{+\infty\},\, (s,y,u )\mapsto \L(s,y,u )$ ($n\ge 1$) is Borel measurable;
\item $\Psi:I\times\R^n\to [0, +\infty]$ is  Borel.
\item    The  {\bf effective domain} of $\L$, given by
   \[\Dom (\L):=\{(s,y,u)\in I\times\R^n\times\R^n:\, \L(s,y,u )<+\infty\}\]
   is of the form $\Dom(\L)=I\times D_{\L}$, with $D_{\L}\subseteq\R^n\times\R^n$.
\end{itemize}
\end{basicass}
\subsection{Notation}
We introduce the main recurring notation:
\begin{itemize}
\item The Euclidean norm of  $x\in \R^n$ is denoted by $|x|$;
\item The Lebesgue measure of a subset $A$ of $I=[t,T]$ is $|A|$ (no confusion may occur with the Euclidean norm);
\item If $y:I\to\R^n$ we denote by $y(I)$ its image, by $\|y\|_{\infty}$ its sup-norm and by $\|y\|_p$ its norm in $L^p(I;\R^n)$;
\item The complement of a set $A$ in $\R^n$ is denoted by $A^c$;
\item The characteristic function of a set $A$ is $\chi_A$.
\item If $x\in \R$, we denote by $x^+$ its positive part, by $x^-$ its negative part;
\item  $\Lip(I;\R^n)=\{y:I\to\R^n,\, y\text{ Lipschitz}\}$;  if $n=1$ we simply write $\Lip(I)$;
\item For $p\ge 1$, $W^{1,p}(I;\R^n)=\{y:I\to\R^n:\, y, y'\in L^p(I;\R^n)\}$; if $n=1$ we simply write $W^{1,p}(I)$.
\end{itemize}
\subsection{Two variational problems}
We shall consider  different variational problems associated to the functional $F$, with different end-point conditions and/or state constraints.
Let $X, Y\in \R^n$ and $\mathcal S\subset\R^n$. We define
\begin{itemize}
\item If $X\in\mathcal S$ we set $\Gamma_X^{\mathcal S}:=\{y\in W^{1,p}(I,\R^n):\, y(t)=X,\, y(I)\subset\mathcal S\}$, and the corresponding variational problem
    \begin{equation}\tag{$\mathcal P_X^{\mathcal S}$}\text{Minimize }\{F(y):\, y\in \Gamma_X^{\mathcal S}\},\end{equation}
    whenever $\inf (\mathcal P_{X}^{\mathcal S})<+\infty$.
\item If $X, Y\in\mathcal S$ we set \[\Gamma_{X, Y}^{\mathcal S}:=\{y\in W^{1,p}(I,\R^n):\, y(t)=X,\,y(T)=Y,\, y(I)\subset\mathcal S\},\] and the corresponding variational problem
\begin{equation}\tag{$\mathcal P_{X,Y}^{\mathcal S}$}\text{Minimize }\{F(y):\, y\in \Gamma_{X,Y}^{\mathcal S}\}.\end{equation}
whenever $\inf (\mathcal P_{X,Y}^{\mathcal S})<+\infty$.
\end{itemize}
For the problems with one end point constraint,  there is no privilege in considering the initial condition $y(t)=X$ instead of the final one $y(T)=Y$: any result obtained here for the above variational problems can be reformulated for  a final end point prescribed variational problems, with the same set of assumptions.
\subsection{Lavrentiev gap at a function and Lavrentiev phenomenon}
In this paper we consider different boundary data for the same integral functional.
\begin{definition}[\textbf{Lavrentiev gap at $y\in W^{1,p}(I;\R^n)$}]\label{defi:Lav}
Let $y\in W^{1,p}(I;\R^n)$ be such that $F(y)<+\infty$ and let $\Gamma\in \{\Gamma_X^{\mathcal S}, \Gamma_{X, Y}^{\mathcal S}\}$.\\ We say that the \textbf{Lavrentiev gap} does not occur at $y$ for the variational problem corresponding to $\Gamma$ if
there exists a sequence $\left(y_h\right)_{h\in\N}$ of functions in $\Lip (I, \R^n)$   satisfying:
 \begin{enumerate}
 \item $\forall h\in\mathbb N\quad y_h\in\Gamma$;
 \item $\displaystyle\limsup_{h\to +\infty}F(y_h)\le F(y)$;
 \item $y_h\to y$ in $W^{1,p}(I;\R^n)$.
\end{enumerate}
We say that the \textbf{Lavrentiev phenomenon} does not occur for the variational problem corresponding to $\Gamma$ if
\begin{equation}\label{tag:Lavrentievdefi}
\inf_{\substack{y\in W^{1,p}(I; \R^n)\\y\in \Gamma}}F(y)=\inf_{\substack{
y\in \Lip(I; \R^n)\\
y\in \Gamma}}F(y).
\end{equation}
\end{definition}
\begin{remark}[Gap and phenomenon] \phantom{AA}
\begin{itemize}
\item Condition 2 in Definition~\ref{defi:Lav} is less restrictive than the one that is usually considered, i.e., that $\displaystyle\lim_{h\to +\infty}F(y_h)= F(y)$. We believe that this one here is more appropriate to describe the Lavrentiev gap.
\item Let $y\in W^{1,p}(I;\R^n)$. The non-occurrence of the phenomenon at $y$ ensures that, given $\varepsilon>0$, there is a {\em Lipschitz} function $\overline y$ satisfying the same boundary data and/or constraints such that $F(\overline y)\le F(y)+\varepsilon$. If $L(s,y,\cdot)$ is convex for all $(s,y)\in I\times\R^n$, and $L(s, \cdot, \cdot)$ is lower semicontinuous, the non-occurrence of the Lavrentiev gap at $y$ implies the convergence of $(y_h)_h$ to $y$ in {\em energy}, i.e.,
    \[\lim_{h\to +\infty}F(y_h)=F(y):\] Indeed in that case  $F$ is weakly lower semicontinuous.
\item Of course, the non-occurrence of the Lavrentiev gap  along a minimizing sequence implies the non-occurrence of the Lavrentiev phenomenon for the same variational problem.
\end{itemize}
\end{remark}
The following celebrated example motivates the need to distinguish problems with just one end point condition from problems with   both end points conditions.
\begin{example}[Manià's example \cite{Mania}]\label{ex:Mania1} Consider the problem of minimizing
\begin{equation}\tag{$\mathcal P_{0,1}$}F(y)=\int_0^1(y^3-s)^2(y')^6\,ds:\, y\in W^{1,1}(I),\, y(0)=0,\, y(1)=1.\end{equation}
Then $y(s):=s^{1/3}$ is  a minimizer and $F(y)=0$. Not only $y$ is not Lipschitz; it turns out (see \cite[\S 4.3]{GBH}) that the Lavrentiev phenomenon occurs, i.e.,
\[0=\min F=F(y)<\inf\{F(y):\, y\in \Lip([0,1]),\, y(0)=0, y(1)=1\}.
\]
However, as it is noticed in \cite{GBH}, the situation changes drastically if one allows to vary the initial boundary condition along the sequence $(y_h)_h$. Indeed it turns out that the sequence $(y_h)_h$, where each $y_h$ is obtained by truncating $y$ at $1/h$, $h\in\mathbb N_{\ge 1}$, as follows:
\[y_h(s):=\begin{cases}{1}/{h^{1/3}}&\text{ if } s\in [0, 1/h],\\
s^{1/3}&\text{ otherwise}, \end{cases}\]
\begin{figure}[h!]
\begin{center}
\includegraphics[width=0.4\textwidth]{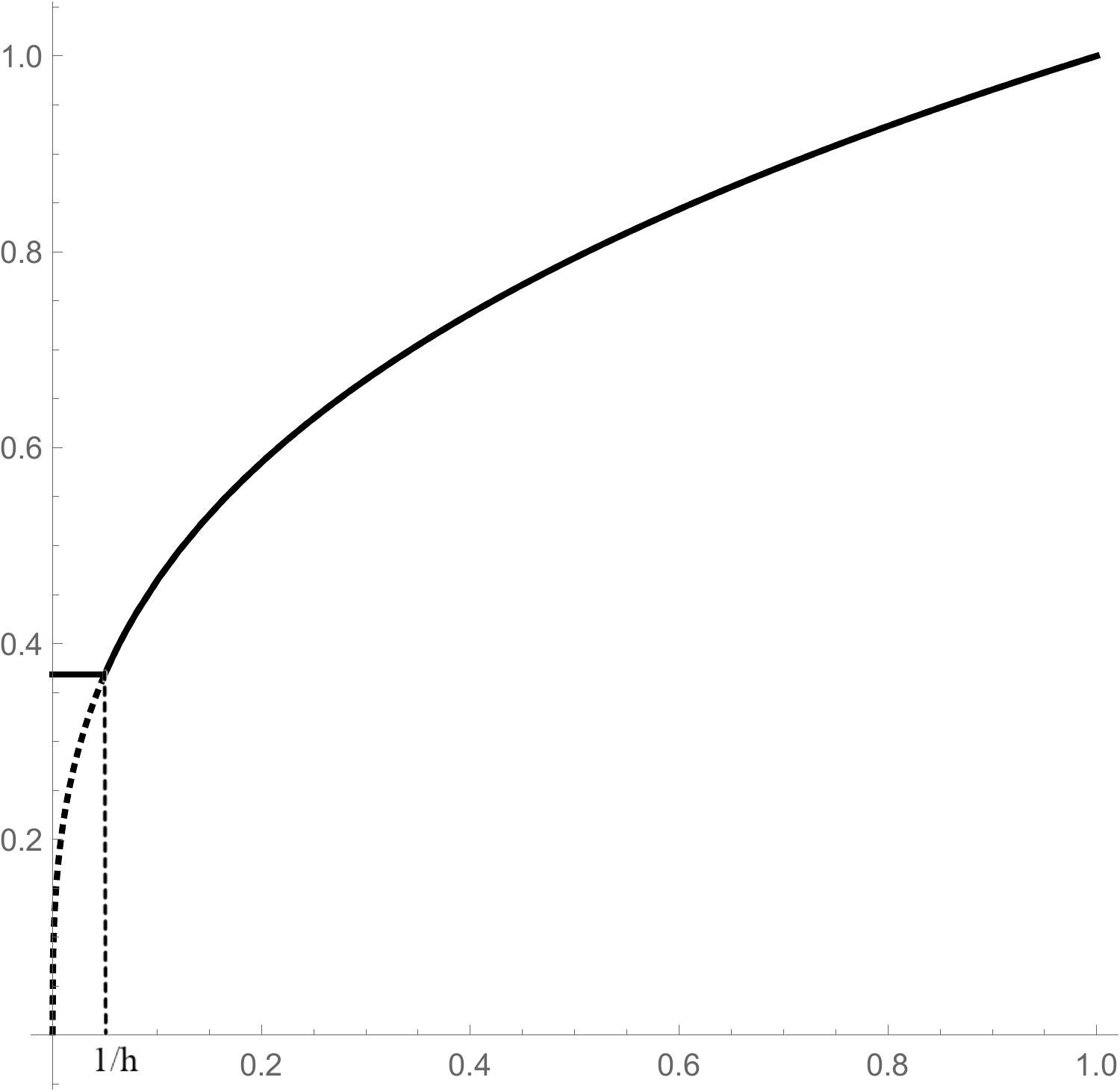}
\caption{{\small The function $y_h$.}}
\end{center}
\end{figure}
is a sequence of Lipschitz functions satisfying
\[y_h(1)=y(1)=1,\quad F(y_h)\to F(y),\quad y_h\to y \text{ in }W^{1,1}([0,1]).\]
Therefore, no Lavrentiev phenomenon occurs for the variational problem
\begin{equation}\min F(y)=\int_0^1(y^3-s)^2(y')^6\,ds:\, y\in W^{1,1}(I),\, \, y(1)=1.\end{equation}
\end{example}
\subsection{Condition (S)}
We consider the following local Lipschitz condition (S) on the first variable of $\L$.
\begin{conditionS}
For every  $K\ge 0$ of $\R^n$ there are $\kappa,  \beta\ge 0, \gamma\in L^1[t,T]$,
$\varepsilon_*>0$ satisfying,
for a.e.  $s\in I$
 \begin{equation}\label{tag:H3}
|\L(s_2,z,v)-\L(s_1,z,v)|\le \big(\kappa \L(s,z,v) +\beta|v|^p+\gamma(s)\big)\,|s_2-s_1|
\end{equation}
whenever $s_1,s_2\in [s-\varepsilon_*,s+\varepsilon_*]\cap I$, $z\in B_K$, $v \in\R^n$, $(s,z,v) \in\Dom(\L)$.
\end{conditionS}
\begin{remark} Condition (S) is fulfilled if $\L=\L(y,u)$ is autonomous. In the smooth setting, Condition (S) ensures the validity of the Erdmann - Du Bois-Reymond (EDBR) condition. In this more general framework it plays a key role in Lipschitz regularity under slow growth conditions \cite{Clarke1993, BM2, MTrans}  and ensures the validity of the (EDBR) for real valued Lagrangians \cite{BM1, BM2}.
\end{remark}
\subsection{Structure assumptions}
We require here some additional Structure Assumptions on $\L$.
\begin{structass}
\begin{itemize}\phantom{AA}
\item[$\bullet$] {(\textbf{Geometry of the effective domain of $\L$})} For every $z\in \R^n$ the set $\{v\in\R^n:\,(z,v)\in D_{\Lambda}\}$   is {strictly star-shaped on the variable $u$ w.r.t.  the origin}, i.e.,
\begin{equation}(z, v)\in D_{\Lambda}, \, 0<r\le 1\Rightarrow (z, rv)\in D_{\Lambda}.\label{tag:starshaped}\end{equation}
\item[$\bullet$] {(\textbf{Radial convexity of $\L$ in the velocity variable})} \text{For a.e. }$s\in I$, for all $(z, v)\in\R^n\times\R^n$,
\begin{equation}\label{tag:R}\tag{A$_c$} 0<r\mapsto\L(s,z,r{v})\text{ is convex}.\end{equation}
\end{itemize}
\end{structass}
\begin{remark}
Assumption (A$c$) implies that, for every $(s,z,v)\in \Dom(\L)$ there is a \textbf{convex subdifferential} for $0<r\mapsto\L(s,z, rv)$ at $r=1$, namely a real number $Q(s,z,v)$ such that
\[\forall r>0\qquad \L(s,z, rv)-\L(s,z,v)\ge Q(s,z,v)(r-1).\]
We shall denote by $\partial_r\L(s,z,rv)_{r=1}$ the set of these subdifferentials, i.e., the \textbf{convex subgradient} of $0<r\mapsto\L(s,z, rv)$ at $r=1$.
It is easy to realize (see, for instance, \cite{VinterBook, BM1}) that Assumption (A$_c$) is equivalent, at every $(s,z,v)\in \Dom(\L)$, of the convexity of the map
\[0<\mu\mapsto \L\left(s, z, \dfrac v{\mu}\right)\mu.\]
In this case, if $P(s,z,v)\in \partial_{\mu}\left[\L\left(s,z,\dfrac v{\mu}\right)\mu\right]_{\mu=1}$, we have
\begin{equation}\label{tag:subinequality}\forall\mu>0\quad \L\left(s, z, \dfrac v{\mu}\right)\mu-\L(s,z,v)\ge P(s,z,v)(\mu-1).\end{equation}
Notice that in the smooth case,
\[P(s,z,v)=\L(s,z,v)-v\cdot \nabla_v\L(s,z,v).\]
\end{remark}
\subsection{Distance-like functions and the compatibility condition (D)}\label{sect:Distance}
In all the paper one can replace a distance-like function with the Euclidean distance $\dist_e$ and Condition (D) with the requirement that $\Dom(\L)$ is open in $I\times\R^n\times\R^n$. It might be convenient, however, to consider other functions than $\dist_e$.
 A distance-like function is a positive function that behaves like the Euclidean distance on  pairs of $I\times\R^n\times\R^n$ having the same two first components.
\begin{definition}[\textbf{Distance-like function $\dist$}] \label{def:dist}Let
$\mathcal W$ be the subset of pairs of elements of $I\times\R^n\times\R^n$ whose two first components coincides. A \textbf{distance-like function} is a function $\dist(\cdot, \cdot)$  with values in $[0, +\infty]$, defined on a suitable symmetric subset of pairs of $I\times\R^n\times\R^n$ containing $\mathcal W$, that coincides with the Euclidean distance on $\mathcal W$, i.e., for all $(s,z)\in I\times\R^n$ and  $v_1, v_2\in\R^n$,
\[\dist((s,z, v_1), (s,z, v_2))= |v_2-v_1|.\]
For all  $(s,z,v)\in I\times \R^n\times\R^n$  and $A\subset I\times \R^n\times\R^n$  we set
\[\dist((s,z,v),A):=
\inf\{\dist((s,z,v), \omega):\, \omega\in  A\}.\]
\end{definition}
Here are some examples of distance-like functions.
\begin{example}[$u$-distance,  Euclidean distance and infinity-distance]\label{ex:distlike}\phantom{AA}
\begin{itemize}
\item We shall denote by $\dist_e$ the usual \textbf{Euclidean distance} in $I\times \R^n\times\R^n$.
\item
$\dist_u$ is the function defined on the pairs of points $\omega_i=(s, z, v_i)\in I\times \R^n\times\R^n, i=1,2$ with {\em the same first two components} by
\[
\dist_{u}(\omega_1, \omega_2)=|v_2-v_1|.\]
\end{itemize}
\end{example}
\begin{remark}\label{rem:triangular} A distance-like function is not necessarily a distance on $I\times\R^n\times\R^n$. For instance   $\dist_u$  is not a distance.
\end{remark}
%
\begin{definition}[Well-inside $\Dom(\L)$ for $\dist$]
 We say that a subset $A$ of $\Dom(\L)$ is
\textbf{well-inside $\Dom(\L)$} w.r.t. a distance-like function $\dist$ if it is contained in  $\{(s,y,u)\in\Dom(\L):\, \dist((s,y,u),\,\Dom(\L)^c)\ge \rho\}$, for a suitable $\rho>0$.
\end{definition}
\begin{example}
Let us examine the property that a set $A$ is  well-inside $\Dom(\L)$ for the distance-like functions $\dist_{\chi}, \chi\in\{e,u\}$ introduced in Example~\ref{ex:distlike}.
\begin{itemize}
\item If $\chi=e$ the above means that for all $(s,y,u)\in A$, the open ball
 of radius $\rho$ in $I\times \R^n\times \R^n$ and center in $(s,y,u)$ is contained in $\Dom(\L)$;
\item If   $\chi=u$ the above means that
\[(s,y,u)\in A,\, 0<r<\rho\Rightarrow (s,y,u+r u)\in \Dom(\L).\]
\item If $\chi=\infty$, {\em any subset} of   $\Dom(\L)$ (even $\Dom(\L)$ itself) is well-inside $\Dom(\L)$.
\end{itemize}
Notice that, if $\omega:=(s,y,u)\in\Dom(\L)$ then
\begin{equation}\label{tag:ineq}\dist_e(\omega, (\Dom(L))^c)\le \dist_{u}(\omega, (\Dom(L))^c).\end{equation}
Thus, if $\mathcal M_{\dist_{\chi}}$  is the class of sets that are well inside $\Dom(\L)$ w.r.t. $\dist_{\chi}$ ($\chi\in\{e,u\}$), we have
\begin{equation}\label{tag:inclusions}
\mathcal M_{\dist_e}\subset \mathcal M_{\dist_u}.
\end{equation}
\end{example}
\begin{example} The inclusions \eqref{tag:inclusions} are strict, in general.
 Let $\L$ be autonomous and $\Dom(\L)=\{(y,u)\in\R^2:\, |y|\le 1\}$. Then the set $\{(y,u)\in\R^2:\, |y|\le 1, |u|\le 1\}$ is well-inside $\Dom(\L)$ w.r.t. to $\dist_u$ but not w.r.t. $\dist_e$.
\end{example}
\begin{example}
Let in $I\times\R\times \R$ $A=I\times (B_1\cup \{0\}\times [0, 2])$. Then, for all $s\in I$,
\[\dist_e((s,0,1), A^c)=0, \quad \dist_u((s,0,1), A^c)=1.\]
\end{example}
We shall impose the following compatibility condition with the effective domain of $\L$.
\begin{conditionD}
A distance-like function $\dist$ satisfies (D) if, for all $(s,z,v)\in\Dom(\L)$ and $\varepsilon>0$ there exists $|v'|\le |v|$ with
\begin{equation}\label{tag:D} (s,z, v')\in\Dom(\L),\quad  \dist((s,z,v'), (\Dom(\L)^c))>0.\end{equation}
\end{conditionD}
\begin{remark} Of course, Condition (D) is satisfied if
\begin{equation}\label{tag:openD}(s,z, v)\in\Dom(\L)\Rightarrow  \dist((s,z,v), (\Dom(\L)^c))>0.\end{equation}
\end{remark}
\begin{example}
Let us consider Condition (D) for the distance-like functions introduced in Example~\ref{ex:distlike}.
\begin{itemize}
\item Regarding the  Euclidean distance $\dist_e$, Condition (D)  is satisfied if $\Dom(\L)$ is open in $I\times\R^n\times\R^n$;
\item If  $\dist=\dist_u$ defined above, Condition (D) is satisfied  whenever for all  $(s,z,v)\in\Dom(\L)$ there are  $|v'|\le |v|$  and  $\rho>0$ such that  $|v'-v''|\ge \rho>0$ for every $(s,z,v'')\in \Dom(\L)^c$. Indeed in this case \[\dist_u((s,z, v'), (\Dom(\L)^c)=\inf\{|v''-v'|:\, (s,z,v'')\in \Dom(\L)^c\}\ge \rho.\] If $n=1$, taking into account the fact that $\Dom(\L)$ is  star-shaped in the last variable,  it is enough to check that $\dist((s, z, 0), (\Dom(\L))^c)>0$ whenever $(s,z,0)\in\Dom(\L)$. Indeed if $(s,z,v)\in \Dom(\L)$ and $v\not=0$ then $(s,z,v')\in\Dom(\L)$ and $\dist_u((s,z,v'), \Dom(\L)^c)\ge\min\{|v-v'|, |v'|\}>0$ for all $v'$ in the relative interior of the segment joining  0 to $v$.
\end{itemize}
\end{example}
\subsection{A useful option: linear growth from below for $\L$}\label{sect:linear}
The following additional linear growth from below on $\L$ is not assumed in the main results; however its validity allows to weaken some of the hypotheses of Corollary~\ref{coro:Lav3} below.
\begin{itemize}
\item  There are $\alpha>0$ and $d\ge 0$ satisfying, for a.e. $s\in[0,T]$ and every $z\in \R^n , v\in \R^n$,
\begin{equation}\label{tag:lingrowth}\tag{{\rm G}$_{\L}$}\L(s,z,{v})\ge \alpha|{v}|-d.\end{equation}
\end{itemize}
\begin{lemma}\cite{CM5}\label{lemma:linbelow}
Let $y\in W^{1,p}(I; \R^n)$ be such that $F(y)<+\infty$. Assume  that $\L$ fulfills ({\rm G}$_{\L}$) and that the infimum of $\Psi$ along the graph of $y$ is strictly positive, i.e.,
\begin{itemize}
\item[(${\rm P}_{y,\Psi}$)] There is $m_{y, \Psi}>0$ such that $\Psi(s, z)\ge m_{y, \Psi}$ for all $s\in I, z\in y(I)$.
\end{itemize}
Then
\[\|y\|_1\le \dfrac{F(y)+\displaystyle m_{y,\Psi}d(T-t)}
{\displaystyle m_{y,\Psi}\alpha}.\]
\end{lemma}
%
\section{Non-occurrence of the Lavrentiev gap/phenomenon for ($\mathcal P_X^{\mathcal S}$) and ($\mathcal P_{X,Y}^{\mathcal S}$)}\label{sect:LavS}
We assume that the infima of both problem ($\mathcal P^{\mathcal S}_X$) and ($\mathcal P^{\mathcal S}_{X, Y}$) is finite, i.e., that each of them has at least an admissible trajectory.
\subsection{Non-occurrence of the Lavrentiev gap}\label{sect:Thm2}
The results of this section make use of the following notion of limit.
\begin{definition} Let $\dist(\cdot, \cdot)$ be a distance-like function and $G\subset I\times \R^n$. We write  that
\begin{equation}\label{tag:limitinfinito}\displaystyle\lim_{\dist((s,z,v),(\Dom (\L))^c)\to 0^+}\L(s,z,v)=+\infty
\text{ unif. w.r.t. }(s, z)\in G, \end{equation}
if
 for all  $r>0$ there exists $\rho>0$ such that, for all $(s,z,v)\in\Dom(\L)$,
  \begin{equation}\label{tag:infinito}\dist((s,z,v),\,(\Dom (\L))^c)\le \rho,\, (s,z)\in G\Rightarrow \L(s,z,v)\ge r.\end{equation}
\end{definition}
\begin{theorem}[\textbf{Non-occurrence of the Lavrentiev  gap for {\rm (}$P_{X}^{\mathcal S}${\rm)} and  {\rm (}$P_{X,Y}^{\mathcal S}${\rm )} at $y\in W^{1,p}(I;\R^n)$}]\label{thm:Lav2} Let $y\in W^{1,p}(I;\R^n)$ be such that
\[\L(s, y(s), y'(s))\Psi(s, y(s))\in L^1(I),\]
and let $B\ge F(y)$. Let $\dist(\cdot, \cdot)$ be a distance-like function satisfying (D).
In addition to the Basic Assumptions and the Structure Assumptions on $\L$ suppose:
\begin{itemize}
  \item[(${\rm B}_{y, \Psi}$)]  $\Psi$ is bounded on $I\times y(I)$;
  \item[(${\rm C}_{y, \Psi}$)] $\Psi(\cdot, z)$ is continuous for every $z\in y(I)$;
\item[(${\rm B}^{w}_{y,\Lambda}$)] There is   $\nu_0>0$ such that $\L$ is bounded on $(I\times y(I)\times B_{\nu_0})\cap \Dom(\L)$;
\item[(${\rm B}'_{y,\Lambda}$)] There is $\lambda>\dfrac{\|y\|_1}{T-t}$
such that $\L$  is {bounded on the subsets} of $I\times y(I)\times B_{\lambda}$ that are well-inside $\Dom(\L)$ w.r.t. $\dist(\cdot, \cdot)$.
   \end{itemize}
   Moreover, assume that  $\L(s, y,y')\in L^1(I)$.
 Then:
\begin{enumerate}[leftmargin=*]
\item
 There is \textbf{no Lavrentiev gap} for ($\mathcal P_{X}^{\mathcal S}$) at $y$.
 \item In addition, assume
 \begin{itemize}
 \item[(${\rm P}_{y,\Psi}$)] There is $m_{y, \Psi}>0$ such that $\Psi(s, z)\ge m_{y,\Psi}$ for all $s\in I, z\in y(I)$.
\end{itemize}
  and that either $\L$ is real valued or
\begin{itemize}
\item[(${\rm L}_{y, \Lambda}$)]
$\displaystyle\lim_{\dist((s,z,v),(\Dom (\L))^c)\to 0^+}\L(s,z,v)=+\infty
\text{ unif. w.r.t. }(s, z)\in I\times y(I)$.
\end{itemize}
 There is \textbf{no Lavrentiev gap} for ($\mathcal P_{X,Y}^{\mathcal S}$) at $y$.
\end{enumerate}
\end{theorem}
\begin{remark}\label{rem:positivity}Notice that in Theorem~\ref{thm:Lav2}, the integrability of $\L(s, y, y')$ is satisfied if $\Psi$ satisfies Condition (${\rm P}_{y,\Psi}$). Indeed, if $\Psi\ge m_{y, \Psi}$ on $I\times y(I)$, then
\[\int_t^T\L(s, y(s), y'(s))\,ds\le \dfrac1{m_{y,\Psi}}F(y)<+\infty.\]
\end{remark}
\begin{remark}[Choice of a suitable distance-like function]\label{rem:dist2}
 The choice of a distance-like function relies on the need of the validity  of Condition (D) (or (${\rm L}_{y, \L}$)) and of (${\rm B}'_{y,\Lambda}$).  Assume that $\dist_j(\cdot, \cdot)$ is a distance-like function defined on a set of triples $\mathcal W_j$ of $I\times\R^n\times\R^n$, $j=1,2$ with \[\forall (s,y,u)\in \Dom(\L)\quad \dist_1((s,y,u), (\Dom(\L))^c)\ge  \dist_2((s,y,u), (\Dom(\L))^c).\] This is the situation, for instance,  if $\dist_1=\dist_u$ and $\dist_2=\dist_e$.
 \begin{itemize}
 \item Since, from \eqref{tag:inclusions}, the sets that are well-inside $\Dom(\L)$ w.r.t. $\dist_2$ are well-inside $\Dom(\L)$ w.r.t. $\dist_1$, the validity of Hypothesis (${\rm B}'_{y,\Lambda}$) w.r.t.  $\dist_1$ implies its validity with  $\dist_2$.
 \item Since, from \eqref{tag:ineq}, for any $(s,z,v')\in\Dom(\L)$ and $\rho>0$,  \begin{equation}\label{tag:ifqgfqi}\dist_2((s,z,v'), (\Dom(\L)^c)\ge \rho\Rightarrow \dist_1((s,z,v'), (\Dom(\L)^c)\ge \rho, \end{equation} then the validity of (D) for $\dist_2$ implies its validity for $\dist_1$.
     \item  Analogously, from \eqref{tag:ifqgfqi}, it follows that the validity of Hypothesis (${\rm L}_{y, \Lambda}$) w.r.t. $\dist_2$ implies its validity w.r.t. $\dist_1$. In particular, there is no way to find a distance-like function for which (${\rm L}_{y, \Lambda}$) is fulfilled if the later is not valid w.r.t. $\dist_u$.
 \end{itemize}
 In particular: For any distance-like function $\dist$  we have
 \[\dist_u((s,y,u), (\Dom(\L))^c)\ge  \dist((s,y,u), (\Dom(\L))^c).\]
 Therefore the validity of (${\rm B}'_{y,\Lambda}$) w.r.t. $\dist_{u}$ implies its validity w.r.t. $\dist$. and if (D) (resp. (${\rm L}_{y, \Lambda}$)) does not hold w.r.t. $\dist_u$ then the property does not hold w.r.t. $\dist$.
     \end{remark}
 \begin{remark}
 The conclusions of  \cite[Theorem 3.1]{CM5} are those of Theorem~\ref{thm:Lav2} when $\mathcal S=\R^n$.
The two theorems do essentially share a same set of assumptions concerning the function $\Psi$.  Concerning $\L$,  both assume
 Assumption {\rm (S)}, which is not technical: The celebrated example by
 Ball and Mizel in \cite{BM} exhibits a positive Lagrangian $\L(s,y,y')$ that is a polynomial, superlinear and convex in $y'$ (thus satisfying all the assumptions of Claim 2 of Theorem~\ref{thm:Lav2} except Condition (S)), for which the Lavrentiev phenomenon occurs for some suitable initial and end boundary data.\\
 However, instead of Conditions (${\rm B}^{w}_{y,\Lambda}$) and (${\rm B}'_{y,\Lambda}$) it is assumed in \cite[Theorem 3.1]{CM5} that:
 \begin{itemize}
 \item[(${\rm B}_{y,\Lambda}$)] There is   $\nu_0>0$ such that $\L$ is bounded on $I\times \mathcal O_y\times B_{\nu_0}$.
 \end{itemize}
 and moreover, for the two end point conditions problem,  instead  of (${\rm L}_{y, \Lambda}$), it is assumed in \cite[Theorem 3.1]{CM5} that
 \begin{itemize}
 \item[(${\rm U}_{y,\L}$)] There is an open subset $U_y$ of $y(I)$ such that, for all $r>0$, $\L$ is bounded on $I\times U_y\times B_r$.
\end{itemize}
 \begin{itemize}
 \item  Hypothesis  (${\rm B}^{w}_{y, \L}$) in Theorem~\ref{thm:Lav2} is less restrictive than  Hypothesis (${\rm B}_{y, \L}$) of \cite[Theorem 3.1]{CM5} and does no more imply that the effective domain of $\L$ contains a rectangle of the  form  $I\times y(I)\times B_{\nu_0}$. As a counterpart, Theorem~\ref{thm:Lav2} requires the additional Hypothesis (${\rm B}'_{y, \L}$).  Figure~1 illustrates the various assumptions in the case of an autonomous Lagrangian with $\Psi\equiv 1$.
 \begin{figure}[h!]
\begin{center}
\includegraphics[width=0.47\textwidth]{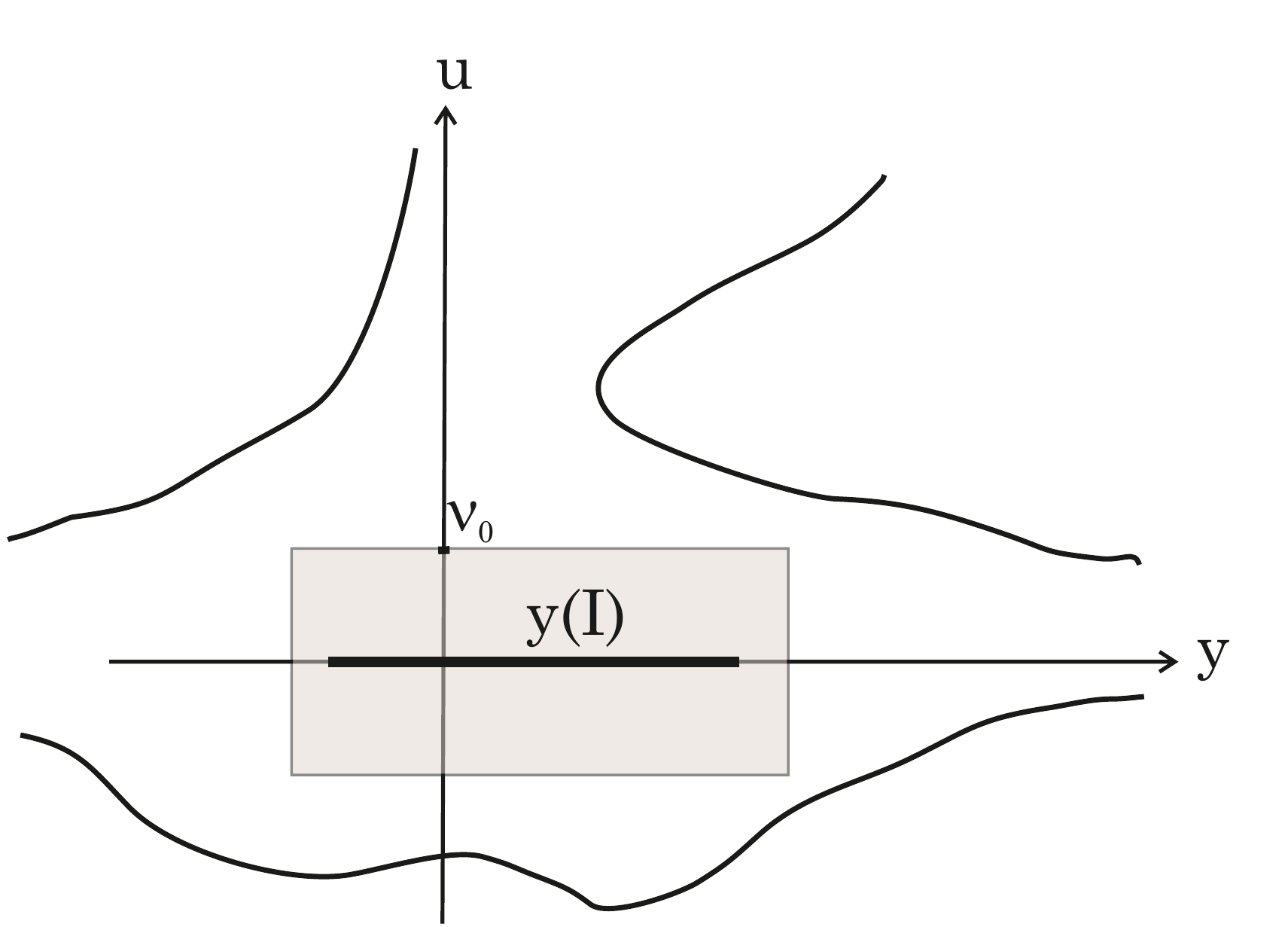}
\includegraphics[width=0.47\textwidth]{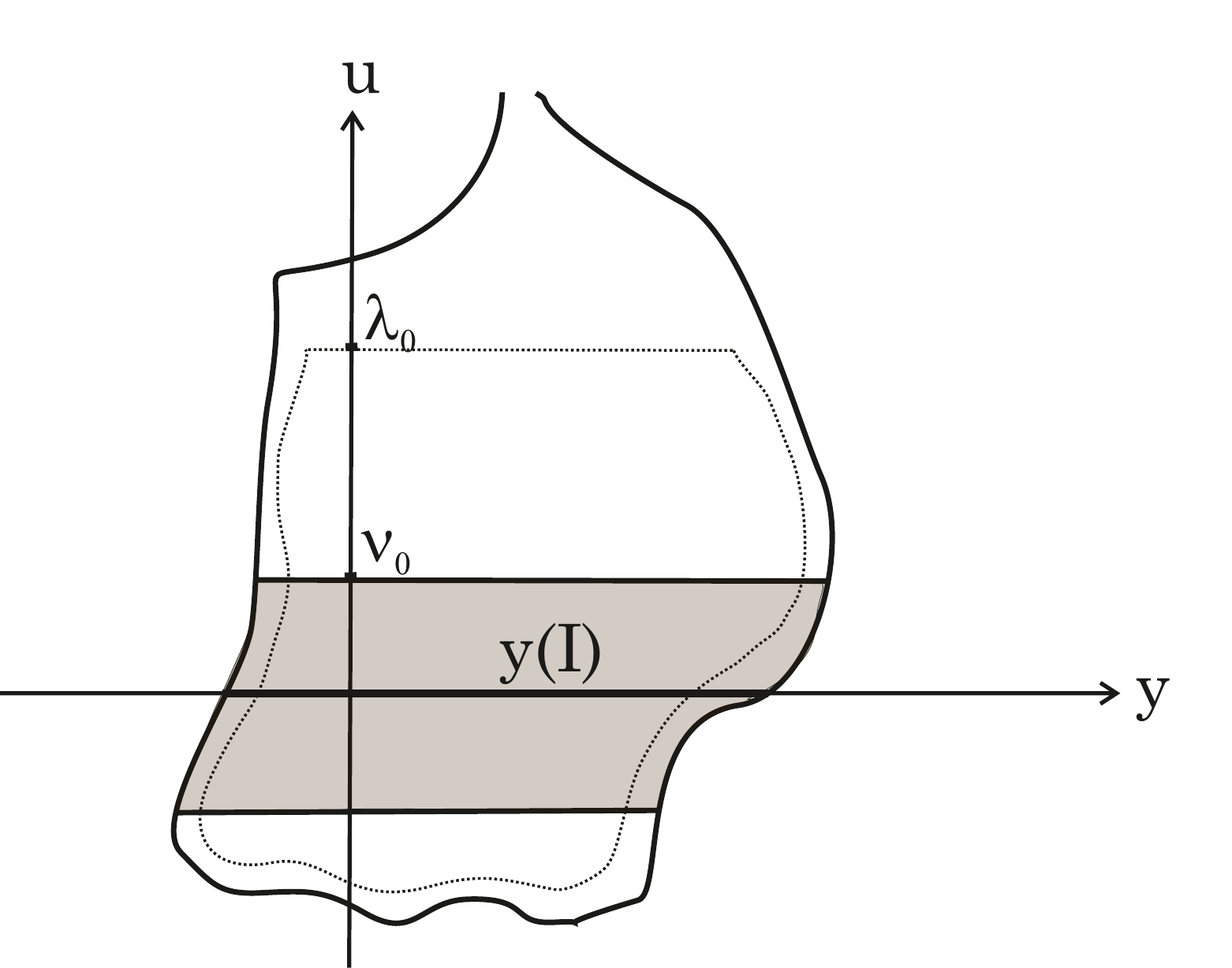}
\caption{{\small The effective domain of an autonomous Lagrangian $\L(y,y')$ ($n=1$, $\dist=\dist_e$) and the validity of the assumptions in Theorem~\ref{thm:Lav2}.
a) Assumption (${\rm B}_{y,\Lambda}$) in \cite[Theorem 3.1]{CM5} requires that  there is ${\nu_0}>0$ such that $\L$ is bounded on  a neighborhhod of $y(I)\times B_{{\nu_0}}$.   b) Hypotheses (${\rm B}^{w}_{y,\Lambda}$) -- (${\rm B}'_{y,\Lambda}$) in Theorem~\ref{thm:Lav2} require that there is a suitable $\nu_0$ such that $\L$ is bounded   on $(y(I)\times B_{\nu_0})\cap \Dom(\L)$ (darker region) and  there is ${\lambda}>\dfrac{\|y\|_1}{T-t}$ such that $\L$ is bounded on  the relatively compact subsets of $(y(I)\times B_{\lambda})\cap \Dom(\L)$  (e.g., the dotted region);  it is also required that $\Dom(\L)$ is star shaped w.r.t. 0 in the last variable.
}}\label{fig:domains}
\end{center}
\end{figure}
\item Hypothesis (${\rm U}_{y,\L}$) implies that $\L$ is finite on the infinite strip
 $I\times U_y\times\R^n$: Functions $\L$ whose effective domain $\L$ is bounded cannot be considered for the two end point conditions problems in \cite[Theorem 3.1]{CM5}.
\item If $\dist=\dist_e$, Hypothesis (${\rm B}'_{y,\Lambda}$) and Condition (D) are satisfied if $\Dom(\L)$ is open in $I\times y(I)\times\R^n$  and  $\L$ is bounded on the relatively compact subsets of $\Dom(\L)$ contained in $I\times y(I)\times \R^n$.
\item  Hypothesis (${\rm L}_{y, \L}$) in Claim 2 of Theorem~\ref{thm:Lav2} implies the validity of  Condition (D). When $\L$ is allowed to take the value $+\infty$, (${\rm L}_{y, \L}$) is not merely a technical assumption: in the absence of (${\rm L}_{y, \L}$), the Lavrentiev phenomenon may occur (see Example~\ref{ex:alberti}).
\end{itemize}
 \end{remark}
\subsection{Non-occorrence of the Lavrentiev phenomenon}
 \begin{corollary}[\textbf{Non-occurrence of the Lavrentiev phenomenon for {\rm (}$P_{X}^{\mathcal S}${\rm)} and for {\rm(}$P_{X,Y}^{\mathcal S}${\rm)}}]\label{coro:Lav3}
 Let $\dist(\cdot, \cdot)$ be a distance-like function satisfying (D). In addition to the Basic Assumptions and the Structure Assumptions on $\L$ suppose, moreover, that for all $K>0$ the following hypotheses hold:
 \begin{itemize}
  \item[(${\rm B}_{\Psi}$)]  $\Psi$ is bounded on $I\times B_K$;
  \item[(${\rm C}_{\Psi}$)] $\Psi(\cdot, z)$ is continuous for every $z\in B_K$;
  \item[(${\rm P}_{\Psi}$)] There is $m_{\Psi}>0$ such that $\Psi(s, z)\ge m_{\Psi}$ for all $s\in I, z\in \R^n$;
\item[(${\rm B}^{w}_{\Lambda}$)] There is   $\nu_0>0$ such that $\L$ is bounded on $(I\times B_K\times B_{\nu_0})\cap \Dom(\L)$;
\item[(${\rm B}'_{\Lambda}$)] For all $\lambda>0$, $\L$  is {bounded on the bounded subsets} of $I\times B_K\times B_{\lambda}$ that are well-inside $\Dom(\L)$ w.r.t. $\dist(\cdot, \cdot)$.
\end{itemize}
Then:
\begin{enumerate}
\item The \textbf{Lavrentiev phenomenon does not occur} for ($\mathcal P_{X}^{\mathcal S}$).
\item Assume, in addition that either $\L$ is real valued or
for all $K>0$,
\begin{itemize}
\item[(${\rm L}_{\Lambda}$)]
$\displaystyle\lim_{\dist((s,z,v),(\Dom (\L))^c)\to 0^+}\L(s,z,v)=+\infty
\text{ unif. w.r.t. }s\in I, z\in B_K$.
\end{itemize}
The \textbf{Lavrentiev phenomenon does not occur} for ($\mathcal P_{X,Y}^{\mathcal S}$).
\item  If, in addition to the assumptions,  \textbf{$\Psi$ satisfies} (${\rm P}_{\Psi}$) and \textbf{$\L$ fulfills} ({\rm G}$_{\L}$), the conclusions of Claims 1,2 hold whenever  Hypotheses (${\rm B}_{\Psi}$), (${\rm C}_{\Psi}$), (${\rm B}^{w}_{\Lambda}$), (${\rm B}'_{\Lambda}$) and (${\rm L}_{\Lambda}$) are satisfied  for just one value of $K>K_0$, where
     \[K_0:=|X|+\dfrac{\inf {\rm(}\mathcal P{\rm)}+\displaystyle m_{\Psi}d(T-t)}
{\displaystyle m_{\Psi}\alpha}, \qquad \mathcal P=\mathcal P_X^{\mathcal S}, \mathcal P_{X, Y}^{\mathcal S}\]
and just one value of $\lambda>\lambda_0$, with
\[\lambda_0:=\dfrac{\inf {\rm(}\mathcal P{\rm )}+\displaystyle m_{\Psi}d(T-t)}
{\displaystyle m_{\Psi}\alpha(T-t)},\qquad \mathcal P=\mathcal P_X^{\mathcal S}, \mathcal P_{X, Y}^{\mathcal S}.\]
\end{enumerate}
 \end{corollary}
 \begin{proof}
 We prove Claim 2 and the part of Claim 3 concerning problem ($\mathcal P_{X, Y}^{\mathcal S}$), the other parts of  Corollary~\ref{coro:Lav3} follow with obvious changes.
 Let $(y_j)_j$ be a minimizing sequence for ($\mathcal P_{X,Y}^{\mathcal S}$) such that
\[\forall j\in\mathbb N\qquad F(y_j)\le \inf {\rm(}\mathcal P_{X,Y}^{\mathcal S}{\rm)}+\dfrac1{j+1}.\]
Fix $j\in\mathbb N$;  the hypotheses of Corollary~\ref{coro:Lav3} with $K:=y_j(I)$ imply the validity of those of Theorem~\ref{thm:Lav2} with $y$ replaced by $y_j$: its application
 yields $\overline y_j\in\Lip(I;\R^n)$
 satisfying the desired boundary conditions and state constraints and, moreover,
 \[F(\overline y_j)\le F(y_j)+\dfrac1{j+1}\le  \inf {\rm(}\mathcal P_{X,Y}^{\mathcal S}{\rm)}+\dfrac2{j+1}.\]
 If $\Psi$ satisfies (${\rm P}_{\Psi}$) and $\L$ satisfies ({\rm G}$_{\L}$), then from Lemma~\ref{lemma:linbelow} for each $j\in\mathbb N$ we have
 \[\|y_j\|_1\le \dfrac{F(y_j)+\displaystyle m_{\Psi}d(T-t)}
{\displaystyle m_{\Psi}\alpha},\]
We may assume, in the proof of Claim 2, that $j$ is big enough in such a way that
\[\begin{aligned}
\dfrac{F(y_j)+\displaystyle m_{\Psi}d(T-t)}
{\displaystyle m_{\Psi}\alpha (T-t)}&\le
\dfrac{\inf {\rm(}\mathcal P_{X,Y}^{\mathcal S}{\rm)}+\frac2{j+1}+\displaystyle m_{\Psi}d(T-t)}
{m_{\Psi}\alpha (T-t)}\\
&\le \lambda_0+\frac2{(j+1)m_{\Psi}\alpha (T-t)}<\lambda
\end{aligned}\]
and
 \[\begin{aligned}\|y_j\|_{\infty}&\le |X|+\|y_j\|_1\\
&\le |X|+\dfrac{\inf {\rm(}\mathcal P_{X,Y}^{\mathcal S}{\rm)}+\frac2{j+1}+ m_{\Psi}d(T-t)}
{m_{\Psi}\alpha}\\
&\le K_0+\frac2{(j+1)m_{\Psi}\alpha}<K,\end{aligned}\]
so that $y_j(I)\subset B_{\overline K}$.
 The claim  follows.\end{proof}
 \begin{remark}
 As noticed in \cite{CM5}, the assumption $\Psi\ge m_{\Psi}>0$ in Claim 2 of Corollary~\ref{coro:Lav3} cannot be weakened to $\Psi\ge 0$, in general.  Indeed, Mania's Example~\ref{ex:Mania1} shows that the Lavrentiev phenomenon with prescribed initial and end conditions may occur when $\Psi\ge 0$ is allowed to take the value 0, even when $\L$ is autonomous.
 \end{remark}
 \begin{remark}\label{rem:hdyroi}
 When $\mathcal S=\R^n$, the conclusions of Corollary~\ref{coro:Lav3} are those of \cite[Corollary 3.6]{CM5}. The hypotheses concerning $\Psi$ do almost overlap, whereas the requirements  on $\L$ differ quite a lot. Indeed, in \cite[Corollary 3.6]{CM5} it is required that:
 \begin{itemize}
 \item[(${\rm B}_{\Lambda}$)] For all $K>0$, there is   $\nu_0>0$ such that $\L$ is bounded on $I\times B_K\times B_{\nu_0}$.
 \end{itemize}
 and, for the two end point conditions problem, that $\L$ is {\em real valued} and
 \begin{itemize}
\item[(U$_{\L}$)]
For all $K, r>0$, $\L$ is bounded on $I\times B_K\times B_r$.
\end{itemize}
It appears that the Hypotheses of Corollary~\ref{coro:Lav3} are more suitable than those of \cite[Corollary 3.6]{CM5} to deal with extended valued Lagrangians and allow functions $\L$ that possess a {\em bounded} effective domain. Indeed:
\begin{itemize}
     \item Hypothesis  (${\rm B}^{w}_{\Lambda}$) in Corollary~\ref{coro:Lav3} is of course fulfilled if (${\rm B}_{\Lambda}$) of  \cite[Corollary 3.6]{CM5} holds. However it is satisfied if just $\L$ is real valued and bounded on the bounded subsets  {\em of}  $\Dom(\L)$.
    \item Hypothesis  (${\rm B}'_{\Lambda}$) is of course equivalent to the fact that $\L$ is  bounded on the bounded subsets  that are well-inside $\Dom(\L)$: Claim 3 of Corollary~\ref{coro:Lav3} motivates the formulation in terms of $K$ and $\lambda$.
\item Differently from Hypothesis (U$_{\L}$) of \cite[Corollary 3.6]{CM5},  Condition (${\rm L}_{\Lambda}$) in Corollary~\ref{coro:Lav3} does not force $\L$ to be real valued.
 \end{itemize}
 \end{remark}
 Many of the assumptions of Corollary~\ref{coro:Lav3} are satisfied for  real valued, continuous Lagrangians, it is worth writing explicitly the result. In this case, the main novelty with respect to Claim 2 in \cite[Corollary 3.6]{CM5} is  the presence of the {\em state constraint} in the variational problem, at the price of radial convexity in the velocity variable.
 \begin{corollary}[\textbf{Non-occurrence of the Lavrentiev phenomenon for {\rm(}$P_{X,Y}^{\mathcal S}${\rm)} -- real valued case}]\label{coro:Lav3real}
 Suppose, in addition to  the Basic Assumptions and the radial convexity (A$_c$) of $\L(s,z,\cdot)$, that $\Psi, \L$ are real valued and:
   \begin{itemize}
   \item $\Psi$ is continuous and stictly positive;
   \item  $\L$ is bounded on bounded sets.
   \end{itemize}
Then the {Lavrentiev phenomenon does not occur} for  for ($\mathcal P_{X,Y}^{\mathcal S}$).
 \end{corollary}
\section{Proof of the main result}
The proof of Theorem~\ref{thm:Lav2} follows the lines of the proof of \cite[Theorem 5.1]{MTrans} where the attention was more focused on the construction of a equi-Lipschitz minimizing sequence, with some uniformity with respect to the initial time and datum. We will emphasize the the new points which are:
\begin{itemize}
\item Condition (${\rm B}'_{y,\L}$)  is new and arises from \cite{BM4} (see Lemma~\ref{lemma:M} below); it weakens the Condition b) in \cite[Proposition 4.24]{MTrans} since the latter requires that $I\times B_K\times B_{\lambda}\subset \Dom(\L)$;
\item The presence of the function $\Psi$;
\item The convergence $y_h\to y$ in $W^{1,1}(I; \R^n)$;
\item Claim 1 of Theorem~\ref{thm:Lav2} and Claim 1 of Corollary~\ref{coro:Lav3} are new even in the autonomous case,  with $\Psi\equiv 1$ and $\L(s,y,u)=L(y,u)$.
\end{itemize}
\subsection{A fundamental Lemma}
The proof of Theorem~\ref{thm:Lav2} relies on the following result.
\begin{lemma}
\cite{BM4, CM5}
\label{lemma:M} Let $\mathcal K$ be a bounded set and let $\dist(\cdot, \cdot)$ be a distance-like function. Assume that:
\begin{itemize}
\item[a)] There is  ${\lambda}>0$ such that $\L$  is {bounded on the subsets of }   $[0,T]\times \mathcal K\times B_{\lambda}$ that are well-inside $\Dom(\L)$ w.r.t. $\dist(\cdot, \cdot)$;
\item[b)] There is $\nu_0>0$ such that $\L$ is bounded on $([0,T]\times  \mathcal K\times B_{\nu_0})\cap\Dom(\L)$.
\end{itemize}
Let, for any $(s,z,v)\in\Dom(\L)$, $P(s,z,v)\in\partial_{\mu}\Big(\L\Big(s,z,\dfrac{{v}}{\mu}\Big){\mu}\Big)_{{\mu}=1}$. Then:
\begin{itemize}
\item[{\em i)}] $P$ is bounded on   the bounded subsets of  $I\times \mathcal K\times B_{\lambda}$ that are well-inside $\Dom(\L)$ w.r.t. $\dist(\cdot, \cdot)$;
\item[{\em ii)}] For all $\rho>0$,
\begin{equation}\label{tag:MMinf}
-\infty<\!\!\!\!\!\!\displaystyle\inf_{\substack{s\in I,z\in \mathcal K, |v|\le \lambda\\ (s,z,v)\in \Dom(\L)\\ \dist((s,z,v),(\Dom{\L})^c)\ge\rho}}
P(s,z,v).
\end{equation}
\item[{\em iii)}] There is  $\nu_0>0$ such that
\begin{equation}\label{tag:MMsup}
\sup_{\substack{s\in I, z\in \mathcal K, |v|\ge \nu_0\\ (s,z,v)\in \Dom(\L)}}
P(s,z,v)<+\infty.
\end{equation}
\end{itemize}
\end{lemma}
The proof of Lemma~\ref{lemma:M} follows narrowly the arguments given in \cite[Lemma 4.18, Proposition 4.24]{MTrans} and the new arguments involved in  \cite[Proposition 3.15]{BM4} in a different framework; we give the full details due to its importance in the proof of Theorem~\ref{thm:Lav2} for the convenience of the reader.
\begin{proof} We will use the fact that $P(s,z,v)=\L(s,z, v)-Q(s,z,v)$, for some $Q(s,z,v)\in \partial_r\L(s, z, ru)_{r=1}$.\\
{\em i)} Let $(s,z,v)\in \Dom(\L)$ and  $Q(s,z,v)\in \partial_{\mu}\L(s, z, ru)_{\mu=1}$. Suppose that, for some $\rho>0$, $z\in \mathcal K$, $|v|\le \lambda$  and $\dist((s,z,v),(\Dom(\L))^c)\ge\rho$. The fact that $\dist(\cdot, \cdot)$ is a distance-like function implies that
\[\dist\left(\left(s,z,v+\dfrac{\rho}{2\lambda}v\right), (\Dom(\L))^c\right)\ge \dfrac{\rho}{2}.\]
Assuming that
\[\partial_r\L(s, z, ru)_{r=1}\not=\emptyset\] we obtain
\[\L\left(s,z,v+\dfrac{\rho}{2\lambda}v\right)-\L(s,z,v)\ge\dfrac{\rho}{2\lambda} Q(s,z,v).\]
The  boundedness assumption of $\L$ implies  that $Q(s,z,v)$ is bounded above by a constant  depending only on $\lambda$ and  $\rho$.
Similarly, from
\[\L\left(s,z,v-\dfrac{\rho}{2\lambda}v\right)-\L(s,z,v)\ge -\dfrac{\rho}{2C} Q(s,z,v),\]
we deduce an upper bound for $Q$.\\
{\em ii)} The set
\[\{(s,z,v)\in\Dom(\L):\, z\in \mathcal K, |v|\le \lambda,  \dist((s,z,v),(\Dom{\L})^c)\ge\rho\}\]
is contained in $[0,T]\times \mathcal K\times B_{\lambda}$ and is well-inside $\Dom(\L)$.
The claim follows immediately from {\em i)}.
\\
{\em iii)}  Let $(s,z,v)\in\Dom(\L)$ with $z\in \mathcal K$ and $|v|\ge {\nu_0}, v\in\mathcal U$.
The assumption that $\Dom(\L)$ is star-shaped in the control variable implies that  \[\Big(s, z, {\nu_0}\dfrac{v}{|v|}\Big)\in\Dom(\L)\] and thus
\begin{equation}
\L\Big(s, z, {\nu_0}\dfrac{v}{|v|}\Big)-\L(s, z, v)\ge Q(s,z,v)\Big(\dfrac{{\nu_0}}{|v|}-1\Big),
\end{equation}
from which we deduce that
\begin{equation}\label{tag:allconvM}
\L(s, z, v)-Q(s,z,v)\le \L\Big(s, z, {\nu_0}\dfrac{v}{|v|}\Big)-\dfrac{{\nu_0}}{|v|}Q(s,z,v).
\end{equation}
The assumptions imply that $\L\Big(s, z, {\nu_0}\dfrac{v}{|v|}\Big)\le C_1(\mathcal K, \nu_0)$ for some  constant $C_1(\mathcal K, \nu_0)$ depending only on $\mathcal K, \nu_0$.
We now provide un upper bound for $-Q(s,z,v)$. Since  $\Big(s, z, \dfrac{\nu_0}2\dfrac{v}{|v|}\Big)\in\Dom(\L)$, then
\begin{equation}\label{tag:qigiug1}\L\Big(s, z, \dfrac{\nu_0}2\dfrac{v}{|v|}\Big)-\L(s, z, v)\ge Q(s,z,v)\left(\dfrac{\nu_0}{2|v|}-1\right),\end{equation}
so that the fact that $\L$ is bounded from below by $-d$  gives
\begin{equation}\label{tag:qigiug2}\begin{aligned}-Q(s,z,v)&\le \dfrac{1}{\left(1-\dfrac{\nu_0}{2|v|}\right)}\left[\L\Big(s, z, \dfrac{\nu_0}2\dfrac{v}{|v|}\Big)-\L(s, z, v)\right]\\
&\le 2\left[\L\Big(s, z, \dfrac{\nu_0}2\dfrac{v}{|v|}\Big)+d\right]
\le C_2(\mathcal K, \nu_0)
\end{aligned}\end{equation}
for some constant $C_2(\mathcal K, \nu_0)$ depending only on $\mathcal K$ and $\nu_0$.
It follows from \eqref{tag:qigiug1} -- \eqref{tag:qigiug2} that the right-hand side of \eqref{tag:allconvM} is bounded above by a constant depending only on $\mathcal K$ and $\nu_0$.
\end{proof}
\subsection{Change of variables and approximations}
We shall often make use of the following change of variables formula for Lebesgue integrals.
\begin{proposition}[\textbf{Change of variables for Lebesgue integrals}]\cite[Corollary 3.16]{Bernal}\label{prop:bernal}
Let $f\ge 0$ be measurable and $\gamma: I\to I$ be bijective, absolutely continuous with $\gamma'>0$ on $I$. Then, for every $A\subset I$,
$f\in L^1(A)\Leftrightarrow  (f\circ\gamma)\gamma'\in L^1(\gamma^{-1}(A))$ and
\[\int_Af(s)\,ds=\int_{\gamma^{-1}(A)}f(\gamma(\tau))\gamma'(\tau)\,d\tau.\]
\end{proposition}
The following approximation argument will be used in the sequel.
\begin{lemma}\label{lemma:approx100}
Let $f\in L^1(I)$ and $(\varphi_{\nu})_{\nu}$ be a sequence of bijiective, absolutely continuous functions $\varphi_{\nu}: I\to I$  with, for all $\nu\in\mathbb N$:
\begin{itemize}
\item  $\varphi_{\nu}'>0$ on $I$,
\item Lipschitz inverse $\psi_{\nu}$;
\item $(\|\psi_{\nu}\|_{\infty})_{\nu}$ bounded;
\item $\varphi_{\nu}(t)\to t$ uniformly.
\end{itemize}
 Then
$f\circ\varphi_{\nu}\to f\text{ in } L^1(I)$.
\end{lemma}
\begin{proof} Consider a sequence $(g_m)_m$ of smooth functions on $I$ such that $g_m\to f$ in $L^1(I)$. For each $m, \nu$ in $\N$ we have
\[\|f\circ \varphi_{\nu}-f\|_1\le \|f\circ \varphi_{\nu}-g_m\circ\varphi_{\nu}\|_1+\|g_m\circ\varphi_{\nu}-g_m\|_1+\|g_m-f\|_1.
\]
Clearly, for each $m$ we have $g_m\circ\varphi_{\nu}\to g_m$ uniformly as $\nu\to +\infty$.
Moreover, from Proposition~\ref{prop:bernal}, the change of variable $\tau=\varphi_{\nu}(s)$ gives
\[\begin{aligned}\|f\circ \varphi_{\nu}-g_m\circ\varphi_{\nu}\|_1&=\int_I|f(\tau)-g_m(\tau)|\psi_{\nu}'(\tau)\,d\tau\\
&\le C\|f-g_m\|_1,\end{aligned}\]
forv a suitable constant $C$. The conclusion follows.
\end{proof}
\subsection{Proof of Theorem~\ref{thm:Lav2}}
Notice first that, in any case,  $\L(s,y,y')\in L^1(I)$ (see Remark~\ref{rem:positivity}).\\
{\em I) Proof of Claim 2.}
We fix $\eta>0$ and prove the existence of a function  $\overline y$ for {\rm (}$\mathcal P_{X,Y}^{\mathcal S}${\rm)} such that
\begin{itemize}
\item[
a)] $\overline y$ is obtained via a reparametrization of $y$ and satisfies the boundary conditions;
\item[b)]  $y'_{\nu}$ is bounded and $\overline y$ is Lipschitz;
\item[c)] $F(\overline y)\le F(y)+\eta$.
\item[d)] $\|\overline y'-y'\|_{L^1(I;\R^n)}\le \eta$
\end{itemize}
\begin{itemize}[leftmargin=*]
\item[{\em i)}]{\em  Definition of $\Xi(\nu)$, $\Upsilon(\rho)$.}\\
Let ${\lambda}$ be as in Hypothesis (${\rm B}'_{y,\Lambda})$.
Let $P(s,z,v)\in \partial_{\mu}\left[\L\Big(s, z, \dfrac{{v}}{\mu}\Big ){\mu}\right]_{{\mu}=1}$.
For $\rho>0$ and  $\nu>0$ we define
\[\Upsilon(\rho):=\displaystyle\inf_{\substack{s\in I,z\in y(I)\\|v|<\lambda\\ (s,z,v)\in\Dom(\L)\\ \dist((s,z,v),(\Dom(\L))^c)\ge\rho}}
P(s,z,v),\quad \Xi(\nu):=\sup_{\substack{s\in I,z\in y(I)\\|v|\ge \nu\\ (s, z,v)\in \Dom(\L)}}
P(s,z,v).\]
We may assume that $\Xi(\nu)>-\infty$ for all $\nu> 0$, otherwise there is $\nu>0$ such that $|y'(s)|\le \nu$ a.e. on $I$ and the conclusion of Theorem~\ref{thm:Lav2} follows trivially.
\item[{\em ii)}]{\em  Choice of $\overline\rho$.}
\begin{itemize}
\item[$\bullet$] There is $\overline\rho>0$ in such a way that
$\Upsilon(\rho)<+\infty$  for all $0<\rho\le \overline\rho$. Indeed it follows easily from Step {\em i)} that the set
$\{s\in I:\, |y'(s)|<\lambda\}$ is non negligible, so that $(s, y(s), y'(s))\in \Dom(\L)$ and $|y'(s)|<\lambda$ for $s$ on a non negligible subset $Z$ of $I$. Here Condition (D) plays its role: for any $s\in Z$ we have $\dist((s, y(s), v'), (\Dom(\L)^c))>0$ for some $|v'|\le |y'(s)|<\lambda$; therefore there is a non negligible subset $Z'$ of  $Z$ and  $\overline \rho>0$ such that, for a.e. $s\in Z'$,
\[\dist ((s, y(s), y'(s)), (\Dom(\L))^c)\ge \overline\rho, \quad |y'(s)|<\lambda.\]
\item[$\bullet$]
It follows from Hypotheses (${\rm B}^{w}_{y,\Lambda}$) -- (${\rm B}'_{y,\Lambda}$) and  Lemma~\ref{lemma:M} that there is $\nu_0>0$ such that
\begin{equation}\label{tag:Lchoicenu500}\forall \rho\in ]0, \overline\rho]\quad\forall\nu\ge\nu_0\qquad
\Upsilon(\rho)\in\R,\quad \,\Xi(\nu_0)\in \R.\end{equation}
\end{itemize}
\item[{\em v)}] {\em For any $\dfrac{\|y\|_1}{\lambda(T-t)}<\mu<1$ let
$\Omega_{\mu}:=\left\{s\in I:\, \dfrac{|y'(s)|}{\mu}<\lambda\right\}.$
Then
\begin{equation}\label{tag:Lestimateomegamu}|\Omega_{\mu}|\ge (T-t)-\dfrac{\|y\|_1}{\mu \lambda}>0.\end{equation}}
Indeed,
$\|y\|_1\ge\displaystyle\int_{[t,T]\setminus \Omega_{\mu}}|y'(s)|\,ds\ge \lambda\mu |[t,T]\setminus \Omega_{\mu}|$.
\item[{\em iv)}] {\em For every $\rho>0$  let
\[I_{\rho}:=\{s\in I:\,\dist((s,y(s), y'(s)), (\Dom(\L))^c)\ge 2\rho\}.\]
Then $\displaystyle\lim_{\rho\to 0}|I_{\rho}|=T-t$.
}
Indeed,
from Hypothesis (${\rm L}_{y,\L}$), there exists $\rho_\ell>0$ satisfying
\begin{equation}\label{tag:Lchoicedelta}\forall s\in I,\, \forall v\in\R^n\,\,  \dist ((s, y(s), v), \partial \Dom(\L))<2\rho_\ell\Rightarrow \L(s,y(s),v)\ge\ell.\end{equation}
Since $\L, \Psi\ge 0$, from Hypothesis (${\rm P}_{y,\Psi}$) we obtain
\begin{equation}\label{tag:748}\int_t^T\L(s, y(s), y'(s))\,ds\le \dfrac1{m_{y,\Psi}}F(y)<+\infty.\end{equation}we  have
\[\begin{aligned}
F(y)&\ge \int_t^T\L(s,y(s),y'(s))\Psi(s, y(s))\,ds\\&\ge
m_{y,\Psi}\,\int_{I\setminus I_{\rho}}\L_j(s,y(s),y'(s))\,ds\ge m_{y,\Psi}\,\ell|I\setminus I_{\rho}|,
\end{aligned}\]
whence
$F(y)/m_{y,\Psi}\ge \ell\big(T-t-|I_{\rho}|\big)$, from which we obtain the estimate
\begin{equation}\label{tag:Lqgiqg}\forall\, 0<\rho\le\rho_{\ell}\qquad |I_{\rho}|\ge \dfrac{\ell(T-t)-F(y)/m_{y,\Psi}}{\ell}\to T-t\text{ as }\ell\to +\infty.\end{equation}
Therefore $\left|I_{\rho_{\ell}}\right|\to (T-t)$ as $\ell\to +\infty$;
the claim follows.
\item[{\em v)}]{\em Let $\overline\rho>0$ be as in Claim ii). There are $\mu=\mu(B,\delta)\in ]0, 1[, \rho\le \overline\rho$, $\Delta\in ]0, 1]$ and a subset $\Omega$ of $\Omega_{\mu}$ with $|\Omega|\ge \Delta(T-t)$, such that, for a.e. $s\in\Omega$
    \begin{equation}\begin{aligned}\label{tag:LNEW!}\left( s, y(s), \dfrac{y'(s)}{\mu}\right)\in \Dom(\L), \,\dfrac{|y'(s)|}{\mu}<\lambda,\\ \dist \left(\left(s, y(s), \dfrac{|y'(s)|}{\mu}\right), (\Dom(\L))^c\right)\ge\rho.\end{aligned}\end{equation}}\\
    See Step {\em vi) } of the proof of \cite[Theorem 5.1]{MTrans}, it is a consequence of Step {\em iv)}. It is essential here that the chosen distance-like function $\dist$ acts as the Euclidean one on the pairs of $\mathcal W$.
\item[{\em vi)}]
\emph {For every $\nu>0$ define \[S_{\nu}:=\{s\in I:\, |y'(s)|> \nu\},\quad \varepsilon_{\nu}:=
\displaystyle\int_{S_{\nu}}\left(\dfrac{|y'(s)|}{\nu}-1\right)\,ds.\]
Then \[|S_{\nu}|\to 0,\quad 0\le \varepsilon_{\nu}\le \dfrac{\|y\|_1}{\nu}\rightarrow 0\text{ as }\nu\to +\infty.\]
}
Indeed,
\[\nu |S_{\nu}|\le\int_{S_{\nu}}|y'(s)|\,ds\le \|y\|_1.\]
\item[{\em vii)}]
{\em Choice of $\nu\ge \nu_0$, of $\Sigma_{\nu}\subseteq \Omega$ and definition of $\Xi$, $\Upsilon, \Theta$.}\\
Taking into account Claim {\em vi)}, we choose $\nu\ge \max\{\nu_0, \lambda\}$  in such a way that \begin{equation}\label{tag:Lchoicenu33}\dfrac{\|y\|_1}{\nu}\le \min\left\{\,(1-\mu)\Delta\,(T-t), \dfrac{\varepsilon_*}{2(1+{\|y\|_{\infty}})}\right\}.\end{equation}
Let \[M_{\Psi}:=\sup\{\Psi(s,z):\, s\in I, \, z\in y(I)\}\]
Notice that, for each $\nu\ge \nu_0$,
\[\Xi(\nu)^+\le \Xi(\nu_0)^+:\]
Let
$ \Upsilon:=\Upsilon(\rho)$, where $\rho$ is as in Step {\em v)} and set
\[\Theta:=2(1+{\|y\|_{\infty}})\left(\|\L(s, y, y')\|_1+\beta \|y'\|_p^p+\|\gamma\|_{\infty}\right).\]
We choose $\nu$ is large enough in such a way that
    \begin{equation}\label{tag:nularge}\dfrac{\|y\|_1}{\nu}M_{\Psi}(\Theta+\Xi(\nu)^++\Upsilon^-)\le \dfrac {{\|y\|_1}}{\nu}M_{\Psi}(\Theta+\Xi(\nu_0)^++\Upsilon^- )\le\dfrac{\eta}{2}, \end{equation}
  From now on we set $\Xi:=\Xi(\nu)$.
Choose a measurable subset ${\Sigma_{\nu}}$ of $ \Omega$ such that $|{\Sigma_{\nu}}|=\dfrac{\varepsilon_{\nu}}{1-\mu}$: this is possible since, from  \eqref{tag:Lchoicenu33} and Step {\em v)}, \[\dfrac{\varepsilon_{\nu}}{1-\mu}\le \Delta(T-\delta)\le  \Delta(T-t)\le |\Omega|.\]
\item[{\em viii)}]
{\em  $S_{\nu}\cap \Omega$ is negligible.}\\
This follows as in Step {\em x)} of the proof of \cite[Theorem 5.1]{MTrans}.
\item[{\em ix)}]
{\em The change of variable $\varphi_{\nu}$.}
We  introduce the following absolutely continuous change of variable $\varphi_{\nu}:I\to \R$ defined by
\begin{equation}\label{tag:newphi}\varphi_{\nu}(t):=t,\quad\text{for a.e. }\tau\in I\quad \varphi_{\nu}'(\tau):=\begin{cases}\dfrac{|y'(\tau)|}{{\nu}} &\text{ if }\tau \in S_{\nu},\\
\,\,\mu &\text{ if }\tau\in{\Sigma_{\nu}},\\
\,\,1 &\text{ otherwise}.
\end{cases}\end{equation}
As  in Step {\em xi)} of the proof of \cite[Theorem 5.1]{MTrans}, $\varphi_{\nu}$ is  strictly increasing and $\varphi_{\nu}:I\to I$ is bijective; let us denote by $\psi_{\nu}$ its inverse, which is  absolutely continuous and  Lipschitz, with $\|\psi_{\nu}'\|_{\infty}\le \dfrac1{\mu}$.
\item[{\em x)}]
{\em Set, for all $s\in [t,T]$,
\begin{equation}\label{tag:yNchange}y_{\nu}(s):=y(\psi_{\nu}(s)).\end{equation}
Then  $y_{\nu}\in W^{1, p}([t,T]; \R^n)$ satisfies the boundary conditions, thus proving a) of the initial claim of the proof.}
  This follows exactly as in Step {\em xii)} of the proof of \cite[Theorem 5.1]{MTrans}.
\item[{\em xi)}] {\em $y_{\nu}\in W^{N+1,\infty}([t,T]; \R^n)$ and $y_{\nu}'$ is  bounded.}\\
Indeed, for all $s\in [t,T]$,
\[y_{\nu}'(s)=\begin{cases} \nu\dfrac{\,y'(\psi_{\nu}(s))}{|y'(\psi_{\nu}(s))|} &\text{ if }s \in \varphi_{\nu}(S_{\nu}),\\
\dfrac{y'(\psi_{\nu}(s))}{\mu} &\text{ if }s\in{\varphi_{\nu}(\Sigma_{\nu})},\\
y'(\psi_{\nu}(s)) &\text{ otherwise.}
\end{cases}\]
Since $|y'(s)|\le  \nu$ a.e. out of $S_{\nu}$ it turns out from the fact that $\Sigma_{\nu}\subseteq\Omega$ that
 \begin{equation}\label{tag:estimnu}|y'_{\nu}(s)|\le
 \max\left\{\nu, \lambda\right\}=\nu.\end{equation}
\item[{\em xii)}]
{\em The following estimate holds:
\[\|\varphi_{\nu}(\tau)-\tau\|_{\infty}\le \int_{t}^{T}\left|\varphi_{\nu}'(s)-1\right|\,ds\le 2\varepsilon_{\nu}.\]}\\
This follows exactly as in Step {\em xiv)} of the proof of \cite[Theorem 5.1]{MTrans}.

\item[{\em xiii)}]
Since $\varphi_{\nu}(s)\to (s, y(s))$ pointwise, from Hypothesis (${\rm C}_{y,\Psi}$) and the fact that $\L(s,y,y')\in L^1(I)$,
we may choose $\nu$ big enough in such a way that
\begin{equation}\label{tag:LPsiNnew}
\left|\int_t^T\L(s, y, y')\left(\Psi(\varphi_{\nu},y)-\Psi(s, y)\right)\,ds\right|\le\dfrac{\eta}{2}.
\end{equation}
\item[{\em xiv)}]
\emph {
A.e. in   $\Omega$,
\begin{equation}\label{tag:LC} \L\left(\varphi_{\nu} ,  y, \dfrac{y'}{\mu}\right)\mu-\L(\varphi_{\nu} ,   y, y')\le -(1-\mu) \Upsilon .\end{equation}}\\
This goes exactly as in Step {\em vii)} of the proof of \cite[Theorem 5.1]{MTrans}.
\item[{\em xv)}]
{\em Estimate of $F(y_{\nu})$ in terms of $\displaystyle\int_{t}^T\L\big(\varphi_{\nu},  y,y'\big)\Psi(\varphi_{\nu}, y)\,d\tau$.}\\
We have
\begin{equation}\label{tag:LJm}F(y_{\nu})=
\int_{t}^T\L(s, y_{\nu}, y_{\nu}')\Psi(s, y_{\nu})\,ds.\end{equation}
 Taking into account \eqref{tag:yNchange}, the change of variables $s=\varphi_{\nu}(\tau), \tau\in I$ yields (in what follows, for brevity, we omit the argument of the functions):
\begin{equation}\begin{aligned}\label{tag:LvaluenewJ}\int_{t}^T\L(s, y_{\nu}, y_{\nu}')\Psi(s, y_{\nu})\,ds&=
\int_{t}^T\left[\L\Big(\varphi_{\nu},   y, \dfrac{y'}{\varphi_{\nu}'}\Big)\varphi_{\nu}'\right]\Psi(\varphi_{\nu},  y)\,d\tau\\
&=\mathcal J_{S_{\nu}}+\mathcal J_{{\Sigma_{\nu}}}+\mathcal J_1^{\nu},\end{aligned}\end{equation}
where we set
\[\begin{aligned}\mathcal J_{S_{\nu}}&:=
\int_{S_{\nu}}\left[\L\Big(\varphi_{\nu},  y, \nu\dfrac{y'}{|y'|}\Big)\dfrac{|y'|}{\nu}\right]\Psi(\varphi_{\nu}, y)\,d\tau,\\
\mathcal J_{{\Sigma_{\nu}}}&:=
\int_{{\Sigma_{\nu}}}\left[\L\Big(\varphi_{\nu},  y,  \dfrac{y'}{\mu}\Big)\mu\right]\Psi(\varphi_{\nu},  y)\,d\tau,\\
 \mathcal J_{1}^{\nu}&:=
\int_{I\setminus ({\Sigma_{\nu}}\cup S_{\nu})} \L\left(\varphi_{\nu},   y,  y'\right)\Psi(\varphi_{\nu},  y)\,d\tau.\end{aligned}\]
The main ingredient here is the subgradient inequality \eqref{tag:subinequality} applied as follows: for every $(s,z,v)\in \Dom(\L)$ and $\mu>0$ such that $\left(s,z, \dfrac{v}{\mu}\right)\in Dom(\L)$ we have
\begin{equation}\label{tag:subinequalitybis}\forall\mu>0\quad \L\left(s, z, \dfrac v{\mu}\right)\mu\le \L(s,z,v)+ P\left(s,z,\dfrac{v}{\mu}\right)(\mu-1).\end{equation}
\begin{itemize} [leftmargin=*]
\item[$\bullet$] {\em Estimate of $\mathcal J_{S_{\nu}}$:}
\begin{equation}\label{tag:LIS}\mathcal J_{S_{\nu}}\le \int_{S_{\nu}}\L\big(\varphi_{\nu},  y, y'\big)\Psi(\varphi_{\nu},  y)\,d\tau+\Xi^+\,M_{\Psi}\varepsilon_{\nu}.\end{equation}
Following  Step {\em xv)} of the proof of \cite[Theorem 5.1]{MTrans}, a.e. in $S_{\nu}$ we obtain
\begin{equation}\label{tag:Lestim45}\L\Big(\varphi_{\nu},   y, \nu\dfrac{y'}{|y'|}\Big)\dfrac{|y'|}{\nu}\le
\L\big(\varphi_{\nu},   y,y'\big)+\Big(\dfrac{|y'|}{\nu}-1
\Big)\Xi.\end{equation}
Notice, in view of the proof of Claim 1,  that the validity of \eqref{tag:Lestim45} does not depend on steps {\em iv) -- v)} and thus it does not rely  on Hypothesis (${\rm L}_{y, \L}$), or on   Hypothesis (${\rm P}_{y,\Psi}$), or on the fact that $\L$ is supposed to be real valued.
We deduce from \eqref{tag:Lestim45} that a.e. in $S_{\nu}$
\begin{multline}\left[\L\Big(\varphi_{\nu},   y, \nu\dfrac{y'}{|y'|}\Big)\dfrac{|y'|}{\nu}\right]\Psi(\varphi_{\nu},  y)\\\le
\L\big(\varphi_{\nu},   y, y'\big)\Psi(\varphi_{\nu},  y)+\Big(\dfrac{|y'|}{\nu}-1
\Big)M_{\Psi}\,\Xi^+,\end{multline}
whence \eqref{tag:LIS}.
\item[$\bullet$] {\em Estimate of $\mathcal J_{{\Sigma_{\nu}}}$.}
Since ${\Sigma_{\nu}}\subseteq \Omega$ and $|{\Sigma_{\nu}}|=\dfrac{\varepsilon_{\nu}}{1-\mu}$,  it is immediate  from \eqref{tag:LC} of Step {\em xiv)}  that
\begin{equation}\label{tag:LISigma}\begin{aligned}\mathcal J_{{\Sigma_{\nu}}}&\le \int_{{\Sigma_{\nu}}}\L\big(\varphi_{\nu},   y, y'\big)\Psi(\varphi_{\nu},  y)\,d\tau+(1-\mu)\Upsilon^- \,M_{\Psi}|{\Sigma_{\nu}}|\\
&\le \int_{{\Sigma_{\nu}}}\L\big(\varphi_{\nu},   y, y'\big)\Psi(\varphi_{\nu},  y)\,d\tau+\Upsilon^- M_{\Psi} {\varepsilon_{\nu}}.\end{aligned}
\end{equation}
\end{itemize}
Therefore, from \eqref{tag:LJm}, \eqref{tag:LvaluenewJ}, \eqref{tag:LIS} and \eqref{tag:LISigma} we deduce the required estimate
\begin{equation}\label{tag:Luqfuqyfdque}\begin{aligned}F(y_{\nu})&
=\left(\mathcal J_{S_{\nu}}+\mathcal J_{{\Sigma_{\nu}}}+\mathcal J_1^{\nu}\right)\\
&\le
\int_{t}^T\L\left(\varphi_{\nu},   y,  y'\right)\Psi(\varphi_{\nu},  y)\,d\tau+\varepsilon_{\nu}M_{\Psi}\left(\Xi^+
+\Upsilon^-  \right).
\end{aligned}\end{equation}
\item[{\em xvi)}]
{\em Estimate of $\displaystyle\int_{t}^T\L\big(\varphi_{\nu},   y,  y'\big)\Psi(\varphi_{\nu},  y)\,d\tau$:}
\begin{equation}\label{tag:Lqigdiuqgdf}
\int_{t}^T\!\L\big(\varphi_{\nu},  y,  y'\big)\Psi(\varphi_{\nu},  y)\,d\tau\le\!
\int_t^T\!\L\big(\tau, y,y'\big)\Psi(\tau, y)\,d\tau+\varepsilon_{\nu}M_{\Psi}\Theta+\dfrac{\eta}{2}.\end{equation}
Indeed,  a.e. on $I$ we have
\[\L\big(\varphi_{\nu}, y,y'\big)\Psi(\tau, y)=Q_{1,\nu}(\tau)+Q_{2,\nu}(\tau)+\L\big(\tau,y,y'\big)\Psi(\tau, y),\]
where
\[Q_{1,\nu}(\tau):=\L\big(\varphi_{\nu},  y, y'\big)\Psi(\varphi_{\nu},  y)-\L\big(\tau, y, y'\big)\Psi(\varphi_{\nu}, y),\]
\[Q_{2,\nu}(\tau)=\L\big(\tau, y, y'\big)\Psi(\varphi_{\nu},  y)-\L\big(\tau, y, y'\big)\Psi(\tau, y).\]
It follows from \eqref{tag:LPsiNnew} that \[\int_t^T|Q_{2,\nu}(\tau)|\,d\tau\le\dfrac{\eta}{2}.\]
Condition {\rm (S)} (with $K:=\|y\|_{\infty}$) and Step {\em xii)} imply that
\[\begin{aligned}\int_t^T|Q_{1,\nu}&(\tau)|\,d\tau\le M_{\Psi}\int_t^T\left|\L\big(\varphi_{\nu},  y, y'\big)-\L\big(\tau, y, y'\big)\right|\,d\tau\\
&\le 2M_{\Psi}(1+{\|y\|_{\infty}})\varepsilon_{\nu}\int_t^T\kappa\L(\tau,  y(\tau),y'(\tau))+\beta|y'(\tau)|+\gamma(\tau)\,d\tau.
\end{aligned}\]
It follows from  Step {\em i)} and the integrability of $\L(s,y, y')$ that
\[\int_t^T|Q_{1,\nu}(\tau)|\,d\tau\le  2M_{\Psi}(1+{\|y\|_{\infty}})\varepsilon_{\nu}\left(\kappa \|\L(s, y, y')\|_1+\beta \|y'\|_p^p+\|\gamma\|_{\infty}\right)\]
which gives \eqref{tag:Lqigdiuqgdf}.
\item[{\em xvii)}]
{\em Final estimate of $F(y_{\nu})$.}
From  \eqref{tag:Luqfuqyfdque} and \eqref{tag:Lqigdiuqgdf} of Steps {\em xv) -- xvi)}, we obtain
\begin{equation}\label{tag:LfinalX}
F(y_{\nu})\le F(y)+\varepsilon_{\nu}M_{\Psi}\left(\Theta+\Xi^+ +\Upsilon^- \right)+\dfrac{\eta}2.
\end{equation}
   The choice of $\nu$ in \eqref{tag:nularge} and the fact that $\varepsilon_{\nu}\le\dfrac{{\|y\|_1}}{\nu}$ yield c).
\item[{\em xviii)}]
{\em We may choose $\nu$ big enough in such a way that $\|y_{\nu}'-y'\|_{L^1(I;\R^n)}\le \eta$.}\\
Indeed,
\begin{equation}\begin{aligned}\label{tag:oqifhoqfhe}
\int_t^T\left|y_{\nu}'-y'\right|^p\,ds&=\int_t^T\left|\dfrac{y'(\psi_{\nu})}{\varphi_{\nu}'(\psi_{\nu})}-y'  \right|^p\,ds\\
&=\int_t^T\left|\dfrac{y'}{\varphi'_{\nu}}-y'(\varphi_{\nu}) \right|^p\varphi'_{\nu}\,d\tau\\
&=\int_{[t,T]\setminus (S_{\nu}\cup \Sigma_{\nu})}*\,d\tau+\int_{\Sigma_{\nu}}*\,d\tau+\int_{S_{\nu}}*\,d\tau
\end{aligned}\end{equation}
where in the above $*$ stands for $\displaystyle \left|\dfrac{y'}{\varphi'_{\nu}}-y'(\varphi_{\nu}) \right|^p\varphi'_{\nu}$.
It follows from  the definition of $\varphi_{\nu}$ in Step {\em ix)} that:
\begin{itemize}[leftmargin=*]
\item[$\bullet$]
\[\begin{aligned}
\int_{[t,T]\setminus (S_{\nu}\cup \Sigma_{\nu})}*\,d\tau&=\int_{[t,T]\setminus (S_{\nu}\cup \Sigma_{\nu})}\left|y'-y'(\varphi_{\nu}) \right|^p\,d\tau\\
&\le \int_t^T\left|y'-y'(\varphi_{\nu}) \right|^p\,d\tau\to 0\quad \nu\to +\infty,
\end{aligned}\]
as a consequence of the fact that,  from Step {\em xi)}, $\|\varphi_{\nu}(\tau)-\tau\|_{\infty}\to 0$ (see Lemma~\ref{lemma:approx100}).
\item[$\bullet$]
\begin{equation}\label{tag:qgdiqgd}\begin{aligned}
\int_{\Sigma_{\nu}}*\,d\tau
&:=\int_{\Sigma_{\nu}}\left|\dfrac{y'}{\mu}-y'(\varphi_{\nu}) \right|^p\varphi'_{\nu}\,d\tau\\
&\le 2^p\left(\dfrac1{\mu^{p-1}}\int_{\Sigma_{\nu}}|y'|^p\,d\tau+\int_{\varphi_{\nu}(\Sigma_{\nu})}|y'|^p\,ds\right).
\end{aligned}\end{equation}
Since $|\Sigma_{\nu}|\to 0$ as $\nu\to +\infty$ then $\displaystyle\int_{\Sigma_{\nu}}|y'|^p\,d\tau\to 0$  and as $\nu\to +\infty$. Moreover, from Steps {\em vi) -- vii)},
\[\begin{aligned}|\varphi_{\nu}(\Sigma_{\nu})|&=\int_{\varphi_{\nu}(\Sigma_{\nu})}1\,ds
=\int_{\Sigma_{\nu}}\varphi'_{\nu}\,d\tau\\
&=\mu|\Sigma_{\nu}|=\dfrac{\mu}{1-\mu}\varepsilon_{\nu}\le \dfrac{\mu \|y\|_1}{\nu(1-\mu)}\to 0\end{aligned}\]
as $\nu\to +\infty$. It follows from \eqref{tag:qgdiqgd} that $\displaystyle\int_{\Sigma_{\nu}}*\,d\tau\to 0$ as $\nu\to +\infty$.
\item[$\bullet$]
\begin{equation}\label{tag:ieugfquiefg}\begin{aligned}
\int_{S_{\nu}}*\,d\tau&:=\int_{S_{\nu}}\left|\dfrac{y'}{\varphi'_{\nu}}-y'(\varphi_{\nu}) \right|^p\varphi'_{\nu}\,d\tau\\
&\le 2^p\left(\nu^p|S_{\nu}|+\int_{\varphi_{\nu}(S_{\nu})}|y'|^p\,ds\right)
\end{aligned}\end{equation}
Now, since $y'\in L^p([t,T])$ and, from Step {\em v)}, $|S_{\nu}|\to 0$ as $\nu\to +\infty$, then
\[\nu^p|S_{\nu}|\le \int_{S_{\nu}}|y'|^p\,d\tau\to 0\quad \nu\to +\infty.\]
Moreover,
\[\begin{aligned}|\varphi_{\nu}(S_{\nu})|&=\int_{\varphi_{\nu}(S_{\nu})}1\,d\tau
=\int_{S_{\nu}}\varphi_{\nu}'\,ds
\\&=\int_{S_{\nu}}\dfrac{|y'|}{\nu}\,ds\le \dfrac{\|y'\|_1}{\nu}\to 0\quad \nu\to +\infty.
\end{aligned}\]
It follows from \eqref{tag:ieugfquiefg} that
$\displaystyle\int_{S_{\nu}}*\,d\tau\to 0$ as $\nu\to +\infty$.
\end{itemize}
Therefore, we deduce from \eqref{tag:oqifhoqfhe} that \[\displaystyle\int_t^T\left|y_{\nu}'-y'\right|^p\,ds\to 0\quad \nu\to +\infty,\]
which concludes the proof of Theorem~\ref{thm:Lav2}.
\end{itemize}
{\em Proof of Claim 1}.  The proof differs slightly from that of Claim 2. As in the proof of \cite[Theorem 3.1]{CM5}, the change of variables  $\varphi_{\nu}$ maps $[t,T]$ onto a bigger interval $\varphi_{\nu}(I)$ containing $I$, and $|\varphi_{\nu}(I)|\to |I|$ as $\nu\to +\infty$.\\
 More precisely, referring to the proof of Claim 2 of Theorem~\ref{thm:Lav2}, we do not need here to introduce the parameter $\mu$ and its related properties formulated in Steps {\em iii), iv), v)}, whose validity depend on the extra assumptions (${\rm P}_{y,\Psi}$),  (${\rm L}_{y,\L}$) or on the fact that $\L$ is real valued.  We just sketch the proof, focusing on the slight differences.
 \begin{itemize}[leftmargin=*]
 \item We keep Steps {\em i)}, {\em ii)}, {\em i)}, {\em ii)}.
 \item We skip Steps {\em iii)}, {\em iv)}, {\em v)}. We set $\rho:=\overline\rho$ defined in Step {\em ii)} and $\Omega:=I$.
 \item We keep Step {\em vi)} and in Step {\em vii)} we choose $\nu$ in such a way that
 \[\dfrac{{\|y\|_1}}{\nu}\le  \dfrac{\varepsilon_*}{2(1+{\|y\|_{\infty}})};\]  we set $\Sigma_{\nu}:=\emptyset$.
 \item Step {\em viii)} now states that $S_{\nu}$ is negligible.
 \item[{\em ix$'$)}]
{\em The change of variable $\varphi_{\nu}$.}
We  introduce the following absolutely continuous change of variable $\varphi_{\nu}:I\to \R$ defined by
\[\varphi_{\nu}(t):=t,\quad\text{for a.e. }\tau\in I\quad \varphi_{\nu}'(\tau):=\begin{cases}\dfrac{|y'(\tau)|}{{\nu}} &\text{ if }\tau \in S_{\nu},\\
\,\,1 &\text{ otherwise}.
\end{cases}\]
Again $\varphi_{\nu}$ is  strictly increasing but now, since $\varphi_{\nu}'\ge 1$   on $I$,  $\varphi_{\nu}(I)$ is an interval containing $I$:  we denote by $\psi_{\nu}$ the restriction of the inverse of $\varphi_{\nu}$ to $I$: $\varphi_{\nu}$ is  absolutely continuous and  Lipschitz, moreover $\psi_{\nu}(I)\subset I$ and $|\psi_{\nu}(I)|\to |I|$ as $\nu\to +\infty$.
\item[{\em x$'$)}]
{\em Set, for all $s\in [t,T]$,
\begin{equation}\label{tag:yNchange}y_{\nu}(s):=y(\psi_{\nu}(s)).\end{equation}
Then  $y_{\nu}\in W^{1, p}([t,T]; \R^n)$ satisfies the boundary condition $y(t)=X$.}
Notice that
\begin{equation}\label{tag:ynuone}y_{\nu}'(s)=\begin{cases} \nu\dfrac{\,y'(\psi_{\nu}(s))}{|y'(\psi_{\nu}(s))|} &\text{ if }\psi_{\nu}(s) \in S_{\nu},\\
y'(\psi_{\nu}(s)) &\text{ otherwise.}
\end{cases}\end{equation}
  The new fact is that now $y_{\nu}(t)=y(\psi_{\nu}(t))=y(t)=X$ but $y_{\nu}(T)=y(\psi_{\nu}(T))=y(t')$ for some $t'\le T$, so that it may happen that $y_{\nu}(T)\not=y(T)$.
\item We now proceed as in the proof of Claim 2, without considering the estimate for $\mathcal J_{{\Sigma_{\nu}}}$ in Step {\em xv)} and of $\displaystyle\int_{\Sigma_{\nu}}*\,d\tau $ in Step {\em xvii)}. Since $\psi_{\nu}(I)\subset I$, we need a little more care in the estimates in the last steps, the change of variable being now $s=\varphi_{\nu}(\tau)$, with $\tau\in \psi_{\nu}(I)\subset I$.
 \item[{\em xv$'$)}]   Instead of  \eqref{tag:LvaluenewJ} we obtain
    \begin{equation}\int_{t}^T\L(s, y_{\nu}, y_{\nu}')\Psi(s, y_{\nu})\,ds=
\int_{\psi_{\nu}(I)}\left[\L\Big(\varphi_{\nu},   y, \dfrac{y'}{\varphi_{\nu}'}\Big)\varphi_{\nu}'\right]\Psi(\varphi_{\nu},  y)\,d\tau.\end{equation}
\item[{\em xvi$'$)}] Instead of \eqref{tag:Lqigdiuqgdf}, one gets
\begin{equation}
\int_{t}^T\!\L\big(\varphi_{\nu},  y,  y'\big)\Psi(\varphi_{\nu},  y)\,d\tau\le\!
\int_{\varphi_{\nu}(I)}\!\L\big(\tau, y,y'\big)\Psi(\tau, y)\,d\tau+\varepsilon_{\nu}M_{\Psi}\Theta+\dfrac{\eta}{2}.\end{equation}
\item[{\em xvii$'$)}]
Therefore
\eqref{tag:LfinalX} is now
\begin{equation}\begin{aligned}
F(y_{\nu})&\le \int_{\psi_{\nu}(I)}\L(s, y, y')\Psi(s,y)\,ds+\varepsilon_{\nu}M_{\Psi}\left(\Theta+\Xi^+ +\Upsilon^- \right)+\dfrac{\eta}2\\
&\le F(y)+\varepsilon_{\nu}M_{\Psi}\left(\Theta+\Xi^+ +\Upsilon^- \right)+\dfrac{\eta}2.
\end{aligned}\end{equation}
\item[{\em xviii$'$)}] Similar arguments apply to the proof of the convergence of $y_{\nu}'$ to $y$ in $L^1(I)$, taking into account that \eqref{tag:oqifhoqfhe} becomes
\begin{equation}\begin{aligned}
\int_t^T\left|y_{\nu}'-y'\right|^p\,ds&=\int_t^T\left|\dfrac{y'(\psi_{\nu})}{\varphi_{\nu}'(\psi_{\nu})}-y'  \right|^p\,ds\\
&=\int_{\psi_{\nu}(I)}\left|\dfrac{y'}{\varphi'_{\nu}}-y'(\varphi_{\nu}) \right|^p\varphi'_{\nu}\,d\tau\\
&=\int_{\psi_{\nu}(I)\setminus S_{\nu}}*\,d\tau+\int_{S_{\nu}}*\,d\tau,
\end{aligned}\end{equation}
where $*$ stands for $\left|\dfrac{y'}{\varphi'_{\nu}}-y'(\varphi_{\nu}) \right|^p\varphi'_{\nu}$.
\phantom{AAAAAAAAAAAAA}$\qed$
 \end{itemize}
\begin{remark} In the case of a final end-point constraint, instead of the initial one, Claim 1 of Theorem~\ref{thm:Lav2} may be obtained by slightly modifying Step {\em xi')} of the proof: indeed it is enough to define
\[\varphi_{\nu}(T):=T,\quad\text{for a.e. }\tau\in I\quad \varphi_{\nu}'(\tau)=\begin{cases}\dfrac{|y'(\tau)|}{{\nu}} &\text{ if }\tau \in S_{\nu},\\
\,\,1 &\text{ otherwise}.
\end{cases}\]
\end{remark}
\begin{remark}In the case of {\em one} end  point constraint, or if $\L$ is real valued, the proof of Theorem~\ref{thm:Lav2} is constructive. Indeed, in the first case, the approximating functions $y_{\nu}$ are defined by $y_{\nu}:=y\circ \psi_{\nu}$, where  $\varphi_{\nu}$ is defined in Step {\em ix$'$)} and depends just on $y'$ and the set $S_{\nu}:=\{s\in I: |y'(s)|>\nu\}$. For problems with both end point constraints and real valued Lagrangians, once chosen $\mu$ and $\Omega_{\mu}$ as in Step {\em v)}, it is enough to choose a subset $\Sigma_{nu}$ of $\Omega_{\mu}$ as in Step {\em vii)}, i.e., in such a way that
$(1-\mu)|\Sigma_{\nu}|=\displaystyle\int_{S_{\nu}}\dfrac{|y'(s)|}{\nu}\,ds$. One then defines the reparametrization $\varphi_{\nu}$ as in Step {\em ix)} and proceeds as above.  Some explicit approximating sequences qe built in   Example~\ref{ex:alberti} and Example~\ref{ex:maniabisbis}.
\end{remark}
\section{Examples}\label{sect:Examples}
\subsection{Autonomous case}
The next examples concern the autonomous case, i.e., $\L=\L(z,v)$ and $\Psi\equiv 1$.
Example~\ref{ex:alberti} below shows that Hypothesis (${\rm L}_{y,\L}$) is essential for the validity of Claim 2 in Theorem~\ref{thm:Lav2}, when $\L$ is extended valued. It
was formulated by G. Alberti (personal communication) for a different purpose.
\begin{example}[Occurrence of the Lavrentiev phenomenon in  an autonomous, convex  and l.s.c. problem with both endpoint constraints]\label{ex:alberti}
Let
 $y\in W^{1,1}([0,1];\R)$ be such that
\begin{itemize}
\item $y$ is of class $C^1$ in $[0,1[$, $y(0)=0, y(1)=1$;
\item $ y'>0$ on $[0, 1[$,
\item $y'(1):=\displaystyle\lim_{s\to 1^-}y'(s)=+\infty$.
    \end{itemize}
Such a function exists,  e.g., $y(s):=1-\sqrt{1-s}, s\in [0,1]$.
For every $z\in [0,1[$ set $q(z):=y'(y^{-1}(z))$. Let
\[\L(s, z,v):=\begin{cases} 0&\text{ if } z\in [0,1[\text{ and } v\le q(z),\\
+\infty&\text{ otherwise},\end{cases}\]
\begin{figure}[h!]
\begin{center}
\includegraphics[width=0.45\textwidth]{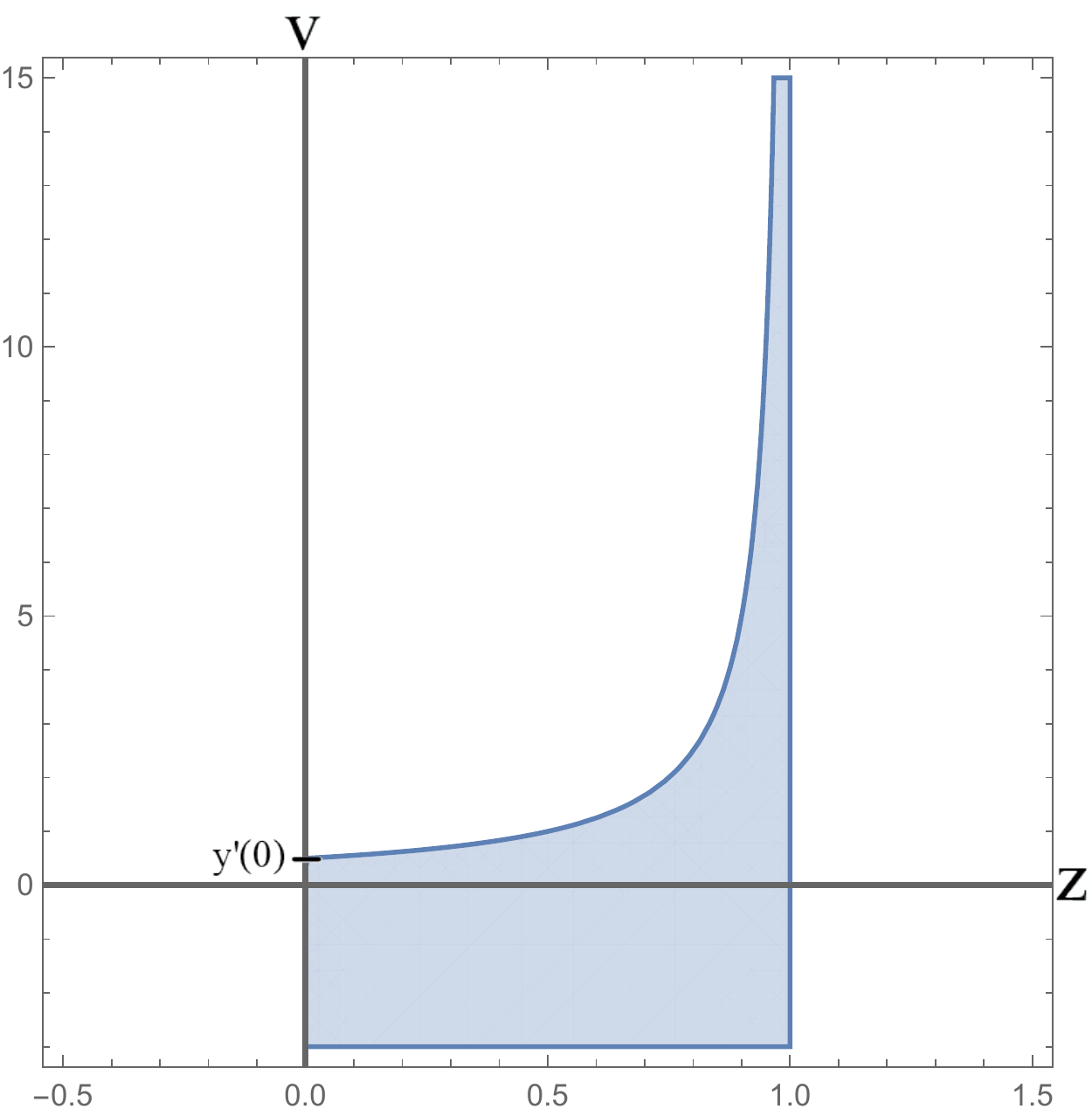}
\caption{{\small The domain of $\L(s, \cdot, \cdot)$ in Example~\ref{ex:alberti}}}
\label{fig:exalberti}\end{center}
\end{figure}
and set $F(z):=\displaystyle\int_0^1\L(s, z(s), z'(s))\,ds$ for every $z\in W^{1,1}([0,1];\R)$.
Notice that $\L$ is lower semicontinuous on $\R^2$ and $\L(z, \cdot)$ is convex for all $z\in\R$.
Clearly $F(y)=\min F=0$. We consider the following points.
\begin{itemize}[leftmargin=*]
\item[a)]
{\em Claim. $F(z)=+\infty$ for every Lipschitz $z:[0,1]\to\R$ satisfying $z(0)=0, z(1)=1$.}
 Indeed assume the contrary: let $z$ be such a function and suppose $F(z)<+\infty$.
Then
\begin{equation}\label{tag:q1}z'(s)\le q(z(s))\text{ a.e.  on } [0,1].\end{equation}
Notice that, since  $\displaystyle\lim_{s\to 1}q(z(s))=+\infty$ and $z'$ is bounded, then necessarily \eqref{tag:q1} is strict on a non negligible set.  It follows that
\[\int_0^1\dfrac{z'(s)}{q(z(s))}\,ds<\int_0^11\,ds=1.\]
However the change of variable $\zeta=z(s)$ (which is justified, for instance, by the chain rule \cite[Theorem 1.74]{MalyZiemer}), gives
\[\begin{aligned}\int_0^1\dfrac{z'(s)}{q(z(s))}\,ds&=\int_0^1\dfrac1{q(\zeta)}\,d\zeta\\
&=\int_0^1\dfrac{1}{y'(y^{-1}(\zeta))}\,d\zeta\\
&= \small{(\tau=y^{-1}(s))} \int_0^1\dfrac{y'(\tau)}{y'(\tau)}\,d\tau=1,
\end{aligned}\]
a contradiction, proving the claim.
\item[b)] {\em Check of the validity of the assumptions of Theorem~\ref{thm:Lav2}.}\\
The Lagrangian here is of the form $\L(z,z')\Psi(s,z)$ with $\Psi\equiv 1$.  Notice that $\Psi$  and $\L(s, y,z)$ satisfy the conditions for the validity of Claim 1  of Theorem~\ref{thm:Lav2} ($\L$ is bounded on its effective domain, in (${\rm B}'_{y,\Lambda}$) take $0<\lambda\le y'(0)$ and is autonomous) and $\Psi$ satisfies the additional Condition (${\rm P}_{y, \Psi}$) of Claim 2.
However $\L$ takes the value $+\infty$ and Condition (${\rm L}_{y, \L}$) {\em is not} fulfilled w.r.t. $\dist_u$, and thus w.r.t.  any other distance-like function (see Remark~\ref{rem:dist2}).
\item[c)] {\em Claim: there is no Lavrentiev phenomenon for $F$ with the end-point constraint $z(1)=1$.}
The validity of Claim 1 of Theorem~\ref{thm:Lav2} implies the non-occurrence of the Lavrentiev phenomenon for the associated variational problems with just one end-point constraint, either $z(0)=y(0)=0$, or $z(1)=y(1)=1$ (but not both!). We point out that this conclusion could not be obtained as a consequence of \cite[Theorem 3.1]{CM5}, since Hypothesis (${\rm B}_{y,\L}$) is violated here: indeed $\L(z,v)$ takes the value $+\infty$  if $z\notin [0,1]$.
\item[d)] {\em Construction of a family of Lipschitz approximating competitors with the end point constraint $z(1)=1$.}
We illustrate here the construction of the ``almost better'' Lipschitz competitor $\overline y=y_{\nu}$ that is carried on in the proof of Theorem~\ref{thm:Lav2} for the problem with final constraint $z(1)=1$.
We assume for simplicity that $y'$ is strictly increasing, as is the case of $y(s):=1-\sqrt{1-s}$.
Let $\nu$ be big enough and let $t_{\nu}\in ]0, 1[$ is such that $y'(t_{\nu})=\nu$; following Step {\em ix$'$)}  of the proof of Theorem~\ref{thm:Lav2},
 the change of variable  $\varphi_{\nu}:[0,1]\to \R$ is defined by
\[\varphi_{\nu}(1):=1,\quad\text{for a.e. }\tau\in [0,1]\quad \varphi_{\nu}'(\tau):=\begin{cases}\dfrac{y'(\tau)}{{\nu}} &\text{ if }\tau\in [t_{\nu}, 1],\\
\,\,1 &\text{ otherwise}.
\end{cases}\]
Therefore we have
\[\varphi_{\nu}(\tau):=\begin{cases}1+\dfrac{y(\tau)-1}{\nu} &\text{ if }\tau\in [t_{\nu}, 1],\\
\,\,\varphi_{\nu}(t_{\nu})+\tau-t_{\nu} &\text{ otherwise},
\end{cases}\]
with $\varphi(t_{\nu})=1+\dfrac{y(t_{\nu})-1}{\nu}$. Notice that $\varphi_{\nu}(0)=\varphi_{\nu}(t_{\nu})-t_{\nu}\le 0$ since $\varphi_{\nu}'\ge 1$ on $[t_{\nu},1]$.
Let $\tau_{\nu}\in [0,1]$ be such that $\varphi_{\nu}(\tau_{\nu})=0$, namely $\tau_{\nu}=t_{\nu}-\varphi_{\nu}(t_{\nu})$. The inverse $\psi_{\nu}$ of $\varphi_{\nu}$, restricted to $[0, 1]$ is thus defined by
\[\psi_{\nu}(s):=\begin{cases}y^{-1}(1-(1-s)\nu) &\text{ if }s\in [\varphi_{\nu}(t_{\nu}), 1],\\
s+t_{\nu}-\varphi_{\nu}(t_{\nu}) &\text{ if } s\in [0, \varphi_{\nu}(t_{\nu})].
\end{cases}\]
Then $\varphi_{\nu}([0,1])=[\tau_{\nu},1]$. The function $y_{\nu}=y\circ\psi_{\nu}$ is thus defined as
\[y_{\nu}(s):=\begin{cases}1+s\nu-\nu &\text{ if }s\in [\varphi_{\nu}(t_{\nu}), 1],\\
y(s+t_{\nu}-\varphi_{\nu}(t_{\nu})) &\text{ if } s\in [0, \varphi_{\nu}(t_{\nu})].
\end{cases}\]
Figure 3 depicts the graphs of some of these approximations for some values of $\nu$ and $y(s):=1-\sqrt{1-s}$.
\begin{figure}[h!]
\begin{center}
\includegraphics[width=0.6\textwidth]{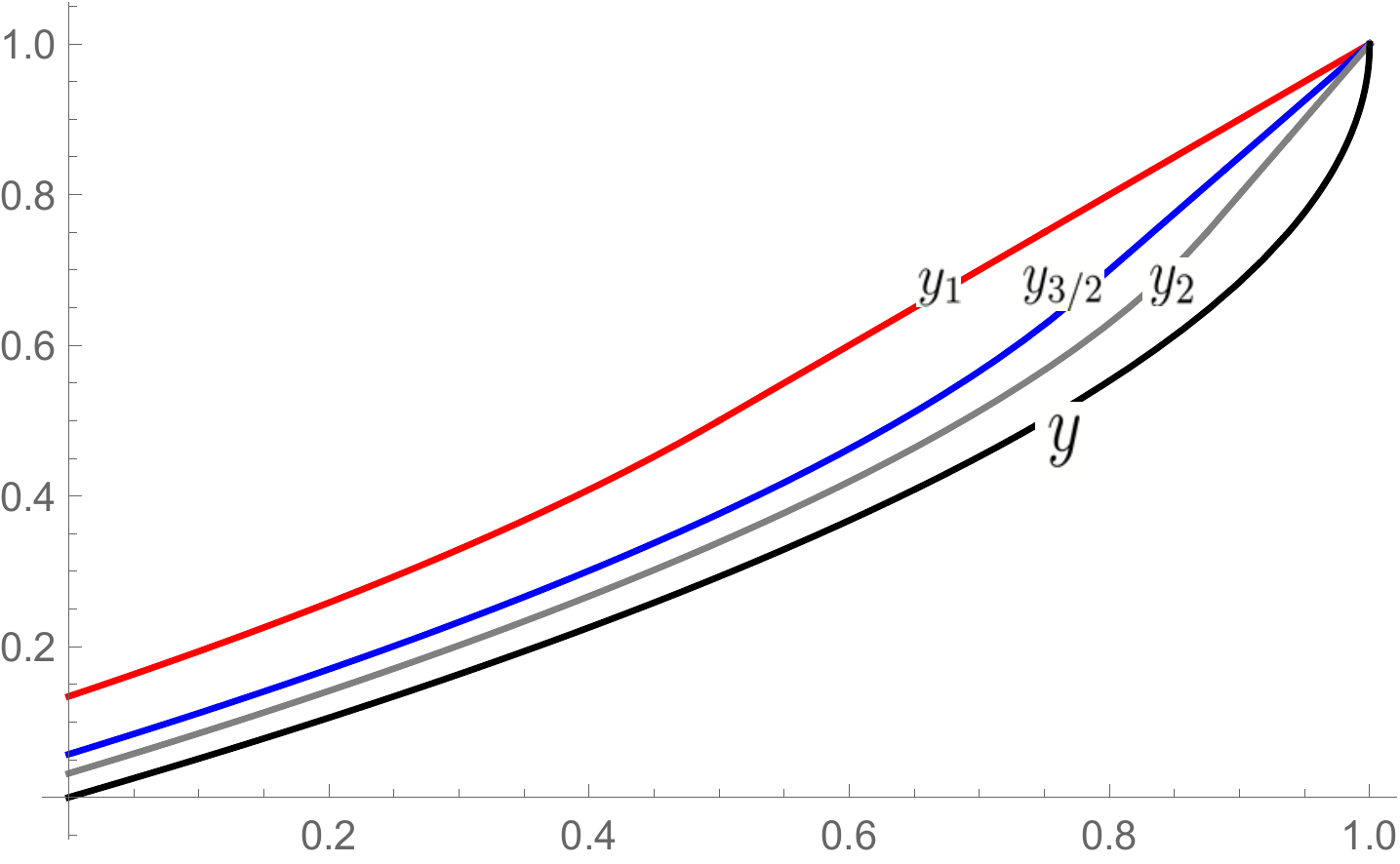}
\caption{{\small The absolutely continuous function $y(s):=1-\sqrt{1-s}$ (below) and some of its Lipschitz approximations (from above: $y_1, y_{3/2}, y_2$), following the recipe of the proof of Theorem~\ref{thm:Lav2}.
}}\label{fig:approx}
\end{center}
\end{figure}
\end{itemize}
\end{example}
In the next example we show how the constructive method in the proof of Theorem~\ref{thm:Lav2} may lead to a sequence $(y_{\nu})_{\nu}$ that converges to $y$ in $W^{1,1}$ but such that $\displaystyle\limsup_{\nu}F(y_{\nu})>F(y)$.
\begin{example}\label{ex:maniabisbis} Consider Manià's example \ref{ex:Mania1}, where the Lagrangian  $L(s,y,u)=(y^3-s)^2u^6$ satisfies the assumptions for the validity of Claim 1 of Theorem~\ref{thm:Lav2}, with $y(s):=s^{1/3}$, except the fact that neither $\inf\Psi(s,y(s))>0$ nor $\L(s, y(s), y'(s))=(y'(s))^6=\dfrac1{3^6}s^{-4}\in L^1([0,1])$. We follow the path of the proof of Theorem~\ref{thm:Lav2} in order to build a Lipschitz sequence $(y_{\nu})_{\nu}$, with prescribed initial datum $y_{\nu}(0)=0$. We will then study the sequence $F(y_{\nu})_{\nu}$.
Fix $\nu\ge 2$. Following Step {\em ix$'$)} we define
\[\varphi_{\nu}(0)=0,\quad \forall\tau\in [0,1]\,\quad\varphi_{\nu}'(\tau):=\begin{cases}\dfrac{|y'(\tau)|}{{\nu}} &\text{ if }y'(\tau)>\nu,\\
\,\,1 &\text{ otherwise}.
\end{cases}\] Thus we obtain
\[\forall \tau\in [0,1]\quad \varphi_{\nu}(\tau):=\begin{cases}\dfrac{\tau^{1/3}}{{\nu}} &\text{ if }0\le \tau<\tau_0(\nu):=\left(\dfrac1{3\nu}\right)^{3/2}\!\!\!\!\!\!\!,\\
\tau+\dfrac{\tau_0(\nu)^{1/3}}{\nu}-\tau_0(\nu) &\text{ if }\tau_0(\nu)\le \tau\le 1.\end{cases}\]
Notice that, as expected,
\[\varphi_{\nu}(1)=1+\dfrac{\tau_0(\nu)^{1/3}}{\nu}-\tau_0(\nu)=
1+\dfrac1{\nu^{3/2}}\left(\dfrac1{\sqrt3}-\dfrac13\right)>1.\]
The inverse $\psi_{\nu}$ of $\varphi_{\nu}$, restricted to $[0,1]$ is thus defined for all $s\in [0,1]$ by
\[\psi_{\nu}(s):=\begin{cases}(\nu s)^3 &\text{ if }0\le s<s_0(\nu):=\varphi_{\nu}(\tau_0(\nu))=\dfrac1{\sqrt 3\nu^{3/2}},\\
 s-\dfrac2{3\sqrt 3\nu^{3/2}}&\text{ if }s_0(\nu)\le s\le 1.\end{cases}\]
Following Step {\em x$'$)} we therefore define, for all $s\in [0,1]$,
\[y_{\nu}(s):=\begin{cases}\nu s &\text{ if }0\le s<s_0(\nu)=\dfrac1{\sqrt {3}\nu^{3/2}},\\
 \left(s-\dfrac2{3\sqrt 3\nu^{3/2}}\right)^{1/3}&\text{ if }s_0(\nu)\le s\le 1.\end{cases}\]
 Notice that of course $y_{\nu}(0)=0$, however $y_{\nu}(1)=\left(1-\dfrac2{3\sqrt 3\nu^{3/2}}\right)^{1/3}<1$, as expected.
The proof of Theorem~\ref{thm:Lav2} ensures that $y_{\nu}$ is Lipschitz for big values of $\nu$ and $(y_{\nu})_{\nu}$  converges to $y$ both in $W^{1,1}$ and in energy. Let us check these facts directly.
\begin{itemize}
\item {\em Lipschitzianity.}
For all $s\in [0,1]$ we have
\[y_{\nu}'(s):=\begin{cases}\nu  \phantom{AAA}\text{ if }0\le s<s_0(\nu)=\dfrac1{\sqrt 3\nu^{3/2}},\\
 y'\left(s-\dfrac2{3\sqrt 3\nu^{3/2}}\right)=\dfrac13\left(s-\dfrac2{3\sqrt 3\nu^{3/2}}\right)^{-2/3}\text{ if }s_0(\nu)\le s\le 1.\end{cases}\]
Now, if $s\ge s_0(\nu)$,
$s-\dfrac2{3\sqrt 3\nu^{3/2}}\ge \dfrac1{3\sqrt 3\nu^{3/2}}$, whence
$\|y_{\nu}'\|_{\infty}\le \nu$ on $[0,1]$.
\item {\em Convergence in $W^{1,1}$.}
We have
\[\begin{aligned}
\|y_{\nu}'-y'\|_1&= \int_0^{s_0(\nu)}|y_{\nu}'-y'|\,ds+\int_{s_0(\nu)}^1|y_{\nu}'-y'|\,ds\\
&\le y_{\nu}(s_0(\nu))+y(s_0(\nu))+\int_{s_0(\nu)}^1|y_{\nu}'-y'|\,ds.
\end{aligned}\]
Now,
\[y_{\nu}(s_0(\nu))+y(s_0(\nu))=\dfrac1{\sqrt{3\nu}}+\dfrac1{3^{1/6}\sqrt{\nu}}\to 0,\]
and, since $y_{\nu}'(s)=y'\left(s-\dfrac2{3\sqrt 3\nu^{3/2}}\right)$ on $[s_0(\nu), 1]$ and $y$ is concave,
\[\begin{aligned}\int_{s_0(\nu)}^1|y_{\nu}'-y'|\,ds&=\int_{s_0(\nu)}^1(y_{\nu}'-y')\,ds\\
&=(y_{\nu}(1)-y(1))-(y_{\nu}(s_0(\nu))-y(s_0(\nu))\\
&=\left(\left(1-\dfrac2{3\sqrt 3\nu^{3/2}}\right)^{1/3}-1\right)-\left(\dfrac1{\sqrt{3\nu}}-\dfrac1{3^{1/6}\sqrt{\nu}}\right)
\end{aligned}\]
tends to 0 as $\nu\to +\infty$.
\item {\em Failure of the convergence in energy.}
Since $\L(s, y,y')\notin L^1([0,1])$ this part of the proof of Theorem~\ref{thm:Lav2} is not justified. Indeed, we have
\[F(y_{\nu})\ge \int_0^1(y_{\nu}^3(s)-s)^2(y_{\nu}'(s))^6\,ds=\frac{68 \nu^{3/2}}{945 \sqrt{3}}\to +\infty\]
as $\nu\to +\infty$.
\end{itemize}\end{example}
\section{Further developments and questions}
\begin{remark}\phantom{AAA}
\begin{enumerate}
\item The proof of Theorem~\ref{thm:Lav2} relies on the fact that, for some $\lambda>\dfrac{\|y\|_1}{T-t}$, $\rho>0$ and $P(s, z, v)\in\partial_r\L(s, z, rv)_{r=1}$,
\begin{equation}\label{tag:MMinfREM}
-\infty<\!\!\!\!\!\!\displaystyle\inf_{\substack{s\in I,z\in y(I), |v|<\lambda\\ (s,z,v)\in \Dom(\L)\\ \dist((s,z,v),(\Dom{\L})^c)\ge\rho}}
P(s,z,v),
\end{equation}
 and there is  $\nu_0>0$ such that
\begin{equation}\label{tag:MMsupREM}
\sup_{\substack{s\in I, z\in y(I), |v|\ge \nu_0\\ (s,z,v)\in \Dom(\L)}}
P(s,z,v)<+\infty.
\end{equation}
This property is called Growth Condition (M$_B^t$) in \cite{MTrans}, and is compared in the quoted paper to other growth conditions.
In view of Lemma~\ref{lemma:M}, Hypotheses (${\rm B}^{w}_{y,\L}$) and (${\rm B}'_{y,\L}$) provide a sufficient condition for the validity of \eqref{tag:MMinfREM} -- \eqref{tag:MMsupREM}.   Can they be weakened?
\item Following the proof of \cite[Theorem 5.1]{MTrans}, it appears that the conclusion of Theorem~\ref{thm:Lav2} is still valid by replacing the radial convexity condition (A$_c$) on the last variable of $\L$ with the existence, at any point $(s,z,v)\in\Dom(\L)$, of the partial derivative $D_v\L(s,z,v)$ of $\L$ with respect to $v$.  In this case one has to replace the selection $P(s,z,v)$ of $\partial_r\L(s,z, rv)_{r=1}$ with $L(s,z,v)-D_v\L(s,z,v)$ (which equals $L(s,z,v)-v\cdot \nabla_v\L(s,z,v)$  if $v\mapsto\L(s,z,v)$ is of class $C^1$). In this framework, however, \eqref{tag:MMinfREM} -- \eqref{tag:MMsupREM} do not follow  for free as in the convex case:  one has to add some further  regularity conditions on $\L$ w.r.t. the last variable, e.g., that $v\mapsto \L(s,z,v)$ is uniformly Lipschitz for $(s,z,v)$ in bounded sets that are well-inside the domain (see \cite[Proposition 4.17]{MTrans}). Are there some other sufficient conditions, other than radial convexity or differentiability, that guarantee the validity of Condition (M$_B^t$) in some suitable form?
\item The conclusions of the paper may be easily extended to Lagrangians of the form $\displaystyle\sum_{i=1}^M\L_i(s,z,v)\Psi_i(s,z)$, assuming that each pair $(\Psi_i, \L_i)$ satisfies the assumptions for $(\Psi, \L)$, and, as in \cite{MTrans}, for optimal control problems with controlled-linear dynamics of the form $z'=b(z)v$ under some suitable assumption on the function $b$. The non-occurrence of the gap at an admissible pair $(y,u)$ means here that the energy at $(y,u)$ may be approximated via the energy of a sequence of admissible pairs $(y_h, u_h)_h$ where each $u_h$ is bounded.  Both of these extensions will be thoroughly described in a forthcoming paper \cite{CM6} devoted to higher order variational problems.
    \end{enumerate}
\end{remark}
\section*{Acknowledgments}
I warmly thank Giovanni Alberti for the mail exchange we had during the preparation of the paper and for providing  Example~\ref{ex:alberti}, though for a different original purpose.
I am also grateful to Giulia Treu for her comments on the manuscript and encouragement.
This research is partially supported by the  Padua University grant SID 2018 ``Controllability, stabilizability and infimum gaps for control systems'', prot. BIRD 187147 and has been accomplished within the UMI Group TAA ``Approximation Theory and Applications''.
\bibliographystyle{plain}
\bibliography{Lavrentiev1dim2}
\end{document}